\documentclass[sort&compress,3p,fleqn]{elsarticle}
\usepackage{mathpazo,MnSymbol,tgpagella,bm}
\usepackage{color,svgcolor}
\usepackage{algorithm,algpseudocode}
\usepackage{amsthm,mathtools}
\numberwithin{equation}{section}
\def\Rii{$\text{R}_{\text{II}}$\ }
\def\diag{\mathop{\mathrm{diag}}}
\newtheorem{theorem}{Theorem}[section]

\theoremstyle{definition}
\newtheorem{example}{Example}
\title{A generalized eigenvalue algorithm for tridiagonal matrix pencils based on a nonautonomous discrete integrable system}

\author[amp]{Kazuki Maeda\corref{cor}}
\ead{kmaeda@amp.i.kyoto-u.ac.jp}

\author[amp]{Satoshi Tsujimoto}
\ead{tujimoto@i.kyoto-u.ac.jp}

\emailauthor{kmaeda@amp.i.kyoto-u.ac.jp}{Kazuki Maeda}
\emailauthor{tujimoto@i.kyoto-u.ac.jp}{Satoshi Tsujimoto}

\cortext[cor]{Corresponding author}
\address[amp]{Department of Applied Mathematics and Physics, Graduate School of Informatics, Kyoto University, Kyoto 606-8501, Japan}

\begin{document}
\begin{abstract}
  A generalized eigenvalue algorithm for tridiagonal matrix pencils is presented.
  The algorithm appears as the time evolution equation of
  a nonautonomous discrete integrable system
  associated with a polynomial sequence which has some orthogonality
  on the support set of the zeros of the characteristic polynomial
  for a tridiagonal matrix pencil.
  The convergence of the algorithm is discussed by using the solution to
  the initial value problem for the corresponding discrete integrable system.
\end{abstract}

\maketitle

\section{Introduction}
Applications of discrete integrable systems to numerical algorithms are
important and fascinating topics.
Since the end of the twentieth century,
a number of relationships between classical numerical algorithms and
integrable systems have been studied
(see the review papers \cite{brezinski2000cad,nakamura2004nan,chu2008laa}).
On this basis, new algorithms based on discrete integrable systems
have been developed:
(i) singular value algorithms for bidiagonal matrices based on the discrete
Lotka--Volterra equation \cite{tsujimoto2001dlv,iwasaki2006acs},
(ii) Pad\'e approximation algorithms based on the discrete relativistic Toda
lattice \cite{minesaki2001drt} and the discrete Schur flow \cite{mukaihira2002sfo},
(iii) eigenvalue algorithms for band matrices based on the discrete hungry Lotka--Volterra
equation \cite{fukuda2009tdh} and the nonautonomous discrete hungry Toda lattice
\cite{fukuda2012slt}, and
(iv) algorithms for computing D-optimal designs based on the nonautonomous
discrete Toda (nd-Toda) lattice \cite{sekido2012acd} and the discrete modified
KdV equation \cite{sekido2012acdt}.

In this paper, we focus on a nonautonomous discrete integrable system
called the \emph{\Rii chain} \cite{spiridonov2000stc}, which is associated
with the generalized eigenvalue problem for tridiagonal matrix pencils
\cite{zhedanov1999brf}.
The relationship between the finite \Rii chain and the generalized eigenvalue
problem can be understood to be an analogue of the connection between
the \emph{finite nd-Toda lattice}
and the eigenvalue problem for tridiagonal matrices.
In numerical analysis, the time evolution equation of the finite nd-Toda lattice
is called the
\emph{dqds (differential quotient difference with shifts) algorithm}
\cite{fernando1994asv},
which is well known as a fast and accurate iterative algorithm for computing
eigenvalues or singular values.
Therefore, it is worth to consider the application of the finite \Rii chain
to algorithms for computing generalized eigenvalues.
The purpose of this paper is to construct a generalized eigenvalue
algorithm based on the finite \Rii chain and to prove the convergence of
the algorithm.
Further improvements and comparisons with traditional methods will be
studied in subsequent papers.

The nd-Toda lattice on a semi-infinite lattice or a non-periodic finite lattice
has a \emph{Hankel determinant solution}.
In the background, there are \emph{monic orthogonal polynomials},
which give rise to this solution;
monic orthogonal polynomials have a determinant expression that relates to the Hankel determinant,
and spectral transformations for monic orthogonal polynomials give
the Lax pair of the nd-Toda lattice
\cite{papageorgiou1995opa,spiridonov1995ddt}.
Especially, for the finite lattice case, we can easily solve the
initial value problem for the nd-Toda lattice with the Gauss quadrature
formula for \emph{monic finite orthogonal polynomials}.
This special property of the discrete integrable system allows us to analyze
the behaviour of the system in detail and tells us how parameters should
be chosen to accelerate the convergence of the dqds algorithm.
We will give a review of this theory in Section~\ref{sec:orth-polyn-nd}.

The theory above will be extended to the \Rii chain in
Section~\ref{sec:rii-polynomials-rii}.
The three-term recurrence relation that monic orthogonal polynomials satisfy
arises from a tridiagonal matrix.
In a similar way, a tridiagonal matrix pencil defines a monic polynomial sequence.
This polynomial sequence, called monic \emph{\Rii polynomials}
\cite{ismail1995goa},
possesses similar properties to monic orthogonal polynomials and their spectral transformations
yield the monic type \Rii chain.
A determinant expression of the monic \Rii polynomials
gives a Hankel determinant solution and, in particular for the finite lattice
case, a convergence theorem of the monic \Rii chain is shown under an assumption.
This theorem enables us to design a generalized eigenvalue algorithm.

The dqds algorithm is a subtraction-free algorithm,
i.e., the recurrence equations of the dqds algorithm do not contain subtraction
operations except origin shifts (see Subsection~\ref{sec:dqds-algorithm}).
The subtraction-free form is numerically effective to avoid the loss of significant digits.
In addition, there is another application of the subtraction-free form:
ultradiscretization \cite{tokihiro1996fse} or tropicalization \cite{itenberg2009tag};
e.g., the ultradiscretization
of the finite nd-Toda lattice in a subtraction-free form
gives a time evolution equation of the box--ball system with a carrier
\cite{maeda2010bbs}.
In Section~\ref{sec:gener-eigenv-algor},
for the monic type \Rii chain, we will present its subtraction-free form,
which contains no subtractions except origin shifts under some conditions.
It is considered that this form makes the computation of the proposed algorithm
more accurate.
At the end of the paper, numerical examples will be presented to confirm that
the proposed algorithm computes the generalized eigenvalues of
given tridiagonal matrix pencils fast and accurately.

\section{Monic orthogonal polynomials, nd-Toda lattice, and dqds algorithm}
\label{sec:orth-polyn-nd}
First, we will review the connection between
the theory of orthogonal polynomials and the nd-Toda lattice.

\subsection{Infinite dimensional case}
\label{sec:infin-dimens-case}

Let us consider a tridiagonal semi-infinite matrix of the form
\begin{equation*}
  B^{(t)}=
  \begin{pmatrix}
    u^{(t)}_0 & 1\\
    w^{(t)}_1 & u^{(t)}_1 & 1\\
    & w^{(t)}_2 & u^{(t)}_2 & 1\\
    && w^{(t)}_3 & \ddots & \ddots\\
    &&& \ddots & \ddots
  \end{pmatrix},\quad
  u^{(t)}_n \in \mathbb C,\quad w^{(t)}_n \in \mathbb C-\{0\},
\end{equation*}
where $t \in \mathbb N$ is the discrete time, whose evolution will be introduced
later.
Let $I_n$ denote the identity matrix of order $n$ and
$B^{(t)}_n$ the $n$-th order leading principal submatrix of $B^{(t)}$.
We now introduce a polynomial sequence $\{\phi^{(t)}_n(x)\}_{n=0}^\infty$:
\begin{equation*}
  \phi^{(t)}_0(x)\coloneq 1,\quad
  \phi^{(t)}_{n}(x)\coloneq \det(xI_n-B^{(t)}_n),\quad n=1, 2, 3, \dots.
\end{equation*}
By definition, $\phi^{(t)}_n(x)$ is a monic polynomial of degree $n$.
The Laplace expansion for $\det(xI_{n+1}-B^{(t)}_{n+1})$ with respect to the last row
yields the three-term recurrence relation
\begin{equation}\label{eq:op-trr}
  \phi^{(t)}_{n+1}(x)=(x-u^{(t)}_n)\phi^{(t)}_n(x)-w^{(t)}_n \phi^{(t)}_{n-1}(x),\quad
  n=0, 1, 2, \dots,
\end{equation}
where we set $w^{(t)}_0\coloneq 0$ and $\phi^{(t)}_{-1}(x)\coloneq 0$.
It is well known that the three-term recurrence relation of the form
\eqref{eq:op-trr} gives the following classical theorem.

\begin{theorem}[{Favard's Theorem \cite[Chapter~I, Section~4]{chihara1978iop}}]
  For the polynomials $\{\phi^{(t)}_n(x)\}_{n=0}^\infty$ satisfying
  the three-term recurrence relation \eqref{eq:op-trr} and any nonzero constant
  $h^{(t)}_0$,
  there exists
  a unique linear functional $\mathcal L^{(t)}$ defined on the space of
  all polynomials such that the orthogonality relation
  \begin{equation}\label{eq:op-orthogonality}
    \mathcal L^{(t)}[x^m \phi^{(t)}_n(x)]=h^{(t)}_n \delta_{m, n},\quad
    n=0, 1, 2, \dots, \quad
    m=0, 1, \dots, n,
  \end{equation}
  holds, where
  \begin{equation*}
    h^{(t)}_n=h^{(t)}_0 w^{(t)}_1 w^{(t)}_2 \dots w^{(t)}_n,\quad
    n=1, 2, 3, \dots,
  \end{equation*}
  and $\delta_{m, n}$ is Kronecker delta.
\end{theorem}

From the relation \eqref{eq:op-orthogonality}, we readily obtain the relation
\begin{equation*}
  \mathcal L^{(t)}[\phi^{(t)}_m(x)\phi^{(t)}_n(x)]=h^{(t)}_n \delta_{m, n},\quad
  m, n=0, 1, 2, \dots.
\end{equation*}
Therefore, the polynomials $\{\phi^{(t)}_n(x)\}_{n=0}^\infty$ are the monic orthogonal polynomials
with respect to $\mathcal L^{(t)}$.

Let us define the \emph{moment} of order $m$ by
\begin{equation*}
  \mu^{(t)}_m \coloneq \mathcal L^{(t)}[x^m],\quad
  m=0, 1, 2, \dots,
\end{equation*}
and its \emph{Hankel determinant} of order $n$ by
\begin{equation*}
  \tau^{(t)}_0\coloneq 1,\quad
  \tau^{(t)}_n\coloneq |\mu^{(t)}_{i+j}|_{i, j=0}^{n-1}=
  \begin{vmatrix}
    \mu^{(t)}_0 & \mu^{(t)}_1 & \dots & \mu^{(t)}_{n-1}\\
    \mu^{(t)}_1 & \mu^{(t)}_2 & \dots & \mu^{(t)}_{n}\\
    \vdots & \vdots & & \vdots \\
    \mu^{(t)}_{n-1} & \mu^{(t)}_n & \dots & \mu^{(t)}_{2n-2}
  \end{vmatrix},\quad
  n=1, 2, 3, \dots.
\end{equation*}
Since the monic orthogonal polynomials with respect to $\mathcal L^{(t)}$
are uniquely determined, we then find the determinant expression of
the polynomial $\phi^{(t)}_n(x)$:
\begin{equation}\label{eq:op-determinant}
  \phi^{(t)}_n(x)=\frac{1}{\tau^{(t)}_n}
  \begin{vmatrix}
    \mu^{(t)}_0 & \mu^{(t)}_1 & \dots & \mu^{(t)}_{n-1} & \mu^{(t)}_n\\
    \mu^{(t)}_1 & \mu^{(t)}_2 & \dots & \mu^{(t)}_{n} & \mu^{(t)}_{n+1}\\
    \vdots & \vdots & & \vdots & \vdots\\
    \mu^{(t)}_{n-1} & \mu^{(t)}_n & \dots & \mu^{(t)}_{2n-2} & \mu^{(t)}_{2n-1}\\
    1 & x & \dots & x^{n-1} & x^n
  \end{vmatrix},\quad n=0, 1, 2, \dots.
\end{equation}

Next, we introduce the discrete time evolution into the monic orthogonal
polynomials by the following transformation
from $\{\phi^{(t)}_n(x)\}_{n=0}^\infty$ to $\{\phi^{(t+1)}_n(x)\}_{n=0}^\infty$:
\begin{equation}\label{eq:op-christoffel}
  (x-s^{(t)})\phi^{(t+1)}_n(x)=\phi^{(t)}_{n+1}(x)+q^{(t)}_n \phi^{(t)}_{n}(x),\quad
  n=0, 1, 2, \dots,
\end{equation}
where
\begin{equation}\label{eq:op-q}
  q^{(t)}_n \coloneq -\frac{\phi^{(t)}_{n+1}(s^{(t)})}{\phi^{(t)}_n(s^{(t)})},\quad
  n=0, 1, 2, \dots,
\end{equation}
and $s^{(t)}$ is a parameter that is not a zero of $\phi^{(t)}_n(x)$
for all $n=0, 1, 2, \dots$.
Suppose that $\{\phi^{(t)}_n(x)\}_{n=0}^\infty$ are the monic orthogonal polynomials with respect to
$\mathcal L^{(t)}$ and define a new linear functional $\mathcal L^{(t+1)}$ by
\begin{equation}\label{eq:op-lf-time-evolution}
  \mathcal L^{(t+1)}[P(x)]\coloneq \mathcal L^{(t)}[(x-s^{(t)}) P(x)]
\end{equation}
for all polynomials $P(x)$.
Then, it is easily verified that $\{\phi^{(t+1)}_n(x)\}_{n=0}^\infty$ are
monic orthogonal polynomials with respect to $\mathcal L^{(t+1)}$ again.
Since it is shown that the monic orthogonal polynomials satisfy the three-term recurrence relation
of the form \eqref{eq:op-trr}, another relation
\begin{equation}\label{eq:op-geronimus}
  \phi^{(t)}_n(x)=\phi^{(t+1)}_n(x)+e^{(t)}_n \phi^{(t+1)}_{n-1}(x),\quad
  n=0, 1, 2, \dots,
\end{equation}
is derived for consistency.
The variable $e^{(t)}_n$ satisfies the compatibility condition
\begin{subequations}\label{eq:nd-Toda-semi-infinite}
  \begin{gather}
    u^{(t+1)}_n=q^{(t+1)}_n+e^{(t+1)}_n+s^{(t+1)}=q^{(t)}_n+e^{(t)}_{n+1}+s^{(t)},\label{eq:nd-Toda-eq1}\\
    w^{(t+1)}_n=q^{(t+1)}_{n-1}e^{(t+1)}_n=q^{(t)}_n e^{(t)}_n\label{eq:nd-Toda-eq2}
  \end{gather}
  with the boundary condition
  \begin{equation}\label{eq:nd-Toda-infinite-bc}
    e^{(t)}_0=0 \quad \text{for all $t \ge 0$}.
  \end{equation}
\end{subequations}
The transformations \eqref{eq:op-christoffel} and \eqref{eq:op-geronimus}
are called
the \emph{Christoffel transformation} and the \emph{Geronimus transformation},
respectively \cite{zhedanov1997rst}.
The discrete dynamical system \eqref{eq:nd-Toda-semi-infinite} is
the semi-infinite nd-Toda lattice.

We have seen above the derivation of the nd-Toda lattice from the theory of orthogonal polynomials.
Using this connection, we can give an explicit solution to the semi-infinite
nd-Toda lattice \eqref{eq:nd-Toda-semi-infinite};
the solution is written in terms of the moments of the
monic orthogonal polynomials.
The time evolution of the linear functional \eqref{eq:op-lf-time-evolution}
leads to
\begin{equation}\label{eq:op-dispersion}
  \mu^{(t+1)}_m=\mu^{(t)}_{m+1}-s^{(t)}\mu^{(t)}_m.
\end{equation}
By applying this relation to the determinant expression of
the monic orthogonal polynomials
\eqref{eq:op-determinant},
the definition of the variable \eqref{eq:op-q} yields
\begin{equation}\label{eq:nd-Toda-sol-q}
  q^{(t)}_n=\frac{\tau^{(t)}_n \tau^{(t+1)}_{n+1}}{\tau^{(t)}_{n+1}\tau^{(t+1)}_n}.
\end{equation}
Further, applying the orthogonality relation \eqref{eq:op-orthogonality} to
equation \eqref{eq:op-geronimus}, we obtain
\begin{equation}\label{eq:nd-Toda-sol-e}
  e^{(t)}_n
  =\frac{\mathcal L^{(t)}[x^n \phi^{(t)}_n(x)]}{\mathcal L^{(t+1)}[x^{n-1}\phi^{(t+1)}_{n-1}(x)]}
  =\frac{h^{(t)}_n}{h^{(t+1)}_{n-1}}
  =\frac{\tau^{(t)}_{n+1}\tau^{(t+1)}_{n-1}}{\tau^{(t)}_n\tau^{(t+1)}_n}.
\end{equation}
If the moments $\mu^{(t)}_m$, the elements of the Hankel determinant
$\tau^{(t)}_n$, are arbitrary functions satisfying the relation
\eqref{eq:op-dispersion}, then
these \eqref{eq:nd-Toda-sol-q} and \eqref{eq:nd-Toda-sol-e} give
particular solutions to the semi-infinite nd-Toda lattice
\eqref{eq:nd-Toda-semi-infinite}.
For instance,
\begin{equation*}
  \mu^{(t)}_m=\int_\Omega x^m \prod_{j=0}^{t-1}(x-s^{(j)}) \omega(x) \, \mathrm{d} x
\end{equation*}
satisfies the relation \eqref{eq:op-dispersion},
where $\Omega$ is an interval of the real line and $\omega(x)$ is
a weight function on $\Omega$.
If the integral of the right-hand side has a finite value for all $m \in \mathbb N$,
then this moment gives a solution.

\subsection{Finite dimensional case}\label{sec:op-finite-dimens-case}

In what follows, we shall reduce the size of the tridiagonal matrix $B^{(t)}$
to finite $N$:
\begin{equation*}
  B^{(t)}\coloneq
  \begin{pmatrix}
    u^{(t)}_0 & 1\\
    w^{(t)}_1 & u^{(t)}_1 & 1\\
    & w^{(t)}_2 & \ddots & \ddots\\
    && \ddots & \ddots & 1\\
    &&& w^{(t)}_{N-1} & u^{(t)}_{N-1}
  \end{pmatrix}.
\end{equation*}
The matrix $B^{(t)}$ determines a system of
\emph{monic finite orthogonal polynomials}.
The corresponding nd-Toda lattice is also reduced to the case of
the non-periodic finite lattice of size $N$:
\begin{subequations}\label{eq:nd-Toda-finite}
  \begin{gather}
    q^{(t+1)}_n+e^{(t+1)}_n+s^{(t+1)}=q^{(t)}_n+e^{(t)}_{n+1}+s^{(t)},\\
    q^{(t+1)}_{n-1} e^{(t+1)}_n=q^{(t)}_{n} e^{(t)}_n,\\
    e^{(t)}_0=e^{(t)}_N=0 \quad \text{for all $t \ge 0$}.
  \end{gather}
\end{subequations}
We can solve the initial value problem for the finite nd-Toda lattice
\eqref{eq:nd-Toda-finite} through the theory of
finite orthogonal polynomials.

The monic finite orthogonal polynomials $\{\phi^{(t)}_n(x)\}_{n=0}^N$ are
defined in the same way as the infinite dimensional case:
$\phi^{(t)}_n(x)\coloneq \det(xI_n-B^{(t)}_n)$.
It should be remarked that $\phi^{(t)}_N(x)$ is
the characteristic polynomial of $B^{(t)}$.
For the polynomials $\{\phi^{(t)}_n(x)\}_{n=0}^N$ and any nonzero constant
$h^{(t)}_0$, there exists a unique linear functional $\mathcal L^{(t)}$
such that the orthogonality relation
\begin{subequations}\label{eq:finite-op-orthogonality}
  \begin{alignat}{2}
    &\mathcal L^{(t)}[x^m \phi^{(t)}_n(x)]=h^{(t)}_n \delta_{m, n},&\quad&
    n=0, 1, \dots, N-1,\quad
    m=0, 1, \dots, n,\\
    \intertext{and the terminating condition}
    &\mathcal L^{(t)}[x^m \phi^{(t)}_N(x)]=0,&& m=0, 1, 2, \dots,\label{eq:op-finite-terminating}
  \end{alignat}
\end{subequations}
hold.

Let $x_0, x_1, \dots, x_{N-1}$ denote the zeros of
the characteristic polynomial $\phi^{(t)}_N(x)$, i.e.,
\begin{equation}\label{eq:op-N-zeros}
  \phi^{(t)}_N(x)=\prod_{i=0}^{N-1}(x-x_i).
\end{equation}
If for simplicity we assume that these zeros are all simple,
the linear functional $\mathcal L^{(t)}$ is concretely given
by the \emph{Gauss quadrature formula}.

\begin{theorem}[{Gauss quadrature formula \cite[Chapter~I, Section~6]{chihara1978iop}}]\label{thm:gqf}
  Let $x_0, x_1, \dots, x_{N-1}$ be the simple zeros of
  the characteristic polynomial $\phi^{(t)}_{N}(x)$.
  For the linear functional $\mathcal L^{(t)}$ of
  the monic finite orthogonal polynomials $\{\phi^{(t)}_n(x)\}_{n=0}^N$,
  there exist some constants $c^{(t)}_0, c^{(t)}_1, \dots, c^{(t)}_{N-1}$
  such that
  \begin{equation}\label{eq:op-gauss-quadrature-formula}
    \mathcal L^{(t)}[P(x)]=\sum_{i=0}^{N-1} c^{(t)}_i P(x_i)
  \end{equation}
  holds for all polynomials $P(x)$.
  Further, if $w^{(t)}_1, w^{(t)}_2, \dots, w^{(t)}_{N-1}$ are
  all real and positive,
  then $c^{(t)}_0, c^{(t)}_1, \dots, c^{(t)}_{N-1}$ are also all real and positive.
\end{theorem}

The formula \eqref{eq:op-gauss-quadrature-formula} means that
the monic finite orthogonal polynomials
$\{\phi^{(t)}_n(x)\}_{n=0}^N$ with the terminating
condition \eqref{eq:op-finite-terminating} are orthogonal on the support set
of the zeros $x_0, x_1, \dots, x_{N-1}$ of
the characteristic polynomial $\phi^{(t)}_N(x)$.

The constants $c^{(t)}_0, c^{(t)}_1, \dots, c^{(t)}_{N-1}$ are calculated
as
\begin{equation}\label{eq:op-discrete-measure}
  c^{(t)}_i=\frac{h^{(t)}_{N-1}}{\phi^{(t)}_{N-1}(x_i){\phi'}^{(t)}_N(x_i)},\quad
  i=0, 1, \dots, N-1,
\end{equation}
where ${\phi'}^{(t)}_N(x)$ is the derivative of $\phi^{(t)}_N(x)$.
This formula is verified as follows.
Due to the Gauss quadrature formula \eqref{eq:op-gauss-quadrature-formula},
the moment is given by
\begin{equation*}
  \mu^{(t)}_m
  =\mathcal L^{(t)}[x^m]
  =\sum_{i=0}^{N-1} c^{(t)}_i x_i^m,\quad
  m=0, 1, 2, \dots.
\end{equation*}
This yields the relation
\begin{equation*}
  \mu^{(t)}_{m+1}-x_j\mu^{(t)}_m=\sum_{\substack{i=0\\ i\ne j}}^{N-1}  c^{(t)}_i(x_i-x_j)x_i^m,\quad
  j=0, 1, \dots, N-1,\quad m=0, 1, 2, \dots.
\end{equation*}
This relation and the determinant expression of
the monic orthogonal polynomials \eqref{eq:op-determinant}
lead to
\begin{equation*}
  \phi^{(t)}_{N-1}(x_j)=\frac{1}{\tau^{(t)}_{N-1}}\prod_{\substack{i=0\\i\ne j}}^{N-1} c^{(t)}_i(x_j-x_i) \prod_{\substack{0 \le \nu_0<\nu_1 \le N-1\\ \nu_0 \ne j, \nu_1 \ne j}}(x_{\nu_1}-x_{\nu_0})^2,\quad
  j=0, 1, \dots, N-1.
\end{equation*}
A similar calculation yields
\begin{equation*}
  \tau^{(t)}_N=\prod_{i=0}^{N-1} c^{(t)}_i\prod_{0 \le \nu_0<\nu_1 \le N-1}(x_{\nu_1}-x_{\nu_0})^2.
\end{equation*}
Further, we have
\begin{gather*}
  {\phi'}^{(t)}_N(x_j)=\prod_{\substack{i=0\\ i\ne j}}^{N-1} (x_j-x_i),\quad
  j=0, 1, \dots, N-1,\\
  h^{(t)}_{N-1}
  =\mathcal L^{(t)}[x^{N-1}\phi^{(t)}_{N-1}(x)]
  =\frac{\tau^{(t)}_N}{\tau^{(t)}_{N-1}}.
\end{gather*}
These equations lead to the formula \eqref{eq:op-discrete-measure}.

The spectral transformations \eqref{eq:op-christoffel} and
\eqref{eq:op-geronimus} also work for the finite dimensional case
except the Christoffel transformation for $n=N$:
\begin{equation}\label{eq:op-christoffel-N}
  \phi^{(t+1)}_N(x)=\phi^{(t)}_N(x),
\end{equation}
which is consistent with the Geronimus transformation for $n=N$.
Equation \eqref{eq:op-christoffel-N} means that
the characteristic polynomial $\phi^{(t)}_N(x)$ of the tridiagonal matrix
$B^{(t)}$ is invariant under the time evolution.
In other words, the time evolution of the monic finite orthogonal polynomials
does not change the eigenvalues of the tridiagonal matrix $B^{(t)}$.
Since the time evolution of the moment is given by \eqref{eq:op-dispersion},
we have the following expression for the moment:
\begin{equation}\label{eq:nd-Toda-finite-moment}
  \mu^{(t)}_m=\sum_{i=0}^{N-1} \left(c^{(0)}_{i} x_i^m \prod_{j=0}^{t-1} (x_i-s^{(j)})\right),
\end{equation}
where we define $t=0$ as the initial time.
Substituting this expression of the moment \eqref{eq:nd-Toda-finite-moment} into
the elements of the Hankel determinant $\tau^{(t)}_n$
and applying the Binet--Cauchy formula and the Vandermonde determinant formula,
we obtain the expanded form of $\tau^{(t)}_n$:
\begin{equation}\label{eq:nd-Toda-finite-hankel}
  \tau^{(t)}_n
  =\sum_{0 \le r_0<r_1<\dots<r_{n-1} \le N-1} \left(\prod_{i=0}^{n-1} \left(c^{(0)}_{r_i} \prod_{j=0}^{t-1} (x_{r_i}-s^{(j)})\right)\prod_{0 \le \nu_0<\nu_1 \le n-1} (x_{r_{\nu_1}}-x_{r_{\nu_0}})^2\right).
\end{equation}
Hence, we can conclude that the solution to the initial value problem for
the finite nd-Toda lattice \eqref{eq:nd-Toda-finite}
is given by
\begin{equation}\label{eq:nd-Toda-finite-qe}
  q^{(t)}_n=\frac{\tau^{(t)}_n\tau^{(t+1)}_{n+1}}{\tau^{(t)}_{n+1}\tau^{(t+1)}_{n}},\quad
  e^{(t)}_n=\frac{\tau^{(t)}_{n+1}\tau^{(t+1)}_{n-1}}{\tau^{(t)}_{n}\tau^{(t+1)}_{n}},
\end{equation}
with the expanded form of $\tau^{(t)}_n$ \eqref{eq:nd-Toda-finite-hankel} and
the expression of $c^{(0)}_n$ \eqref{eq:op-discrete-measure}.

In the rest of this subsection, we will reformulate the matrix forms of
the finite nd-Toda lattice.
Let $L^{(t)}$ and $R^{(t)}$ be bidiagonal matrices of order $N$:
\begin{equation*}
  L^{(t)}=
  \begin{pmatrix}
    1\\
    e^{(t)}_1 & 1\\
    & e^{(t)}_2 & \ddots\\
    && \ddots & \ddots\\
    &&& e^{(t)}_{N-1} & 1
  \end{pmatrix},\quad
  R^{(t)}=
  \begin{pmatrix}
    q^{(t)}_0 & 1\\
    & q^{(t)}_1 & 1\\
    && \ddots & \ddots\\
    &&& \ddots & 1\\
    &&&& q^{(t)}_{N-1}
  \end{pmatrix},
\end{equation*}
and let $\bm\phi^{(t)}(x)$ and $\bm \phi^{(t)}_N(x)$ be the vectors of order $N$:
\begin{equation*}
  \bm\phi^{(t)}(x)=
  \begin{pmatrix}
    \phi^{(t)}_0(x)\\
    \phi^{(t)}_1(x)\\
    \vdots\\
    \phi^{(t)}_{N-1}(x)
  \end{pmatrix},\quad
  \bm\phi^{(t)}_{N}(x)=
  \begin{pmatrix}
    0\\
    \vdots\\
    0\\
    \phi^{(t)}_{N}(x)
  \end{pmatrix}.
\end{equation*}
Then, the three-term recurrence relation \eqref{eq:op-trr} and
the spectral transformations \eqref{eq:op-christoffel} and
\eqref{eq:op-geronimus} are written as
\begin{subequations}\label{eq:nd-Toda-finite-lax}
  \begin{gather}
    B^{(t)}\bm \phi^{(t)}(x)+\bm\phi^{(t)}_N(x)=x\bm \phi^{(t)}(x),\label{eq:nd-Toda-finite-ev}\\
    (x-s^{(t)})\bm \phi^{(t+1)}(x)=R^{(t)}\bm \phi^{(t)}(x)+\bm\phi^{(t)}_N(x),\label{eq:op-christoffel-matrix}\\
    \bm\phi^{(t)}(x)=L^{(t)}\bm\phi^{(t+1)}(x).\label{eq:op-geronimus-matrix}
  \end{gather}
\end{subequations}
Note that, from \eqref{eq:op-N-zeros} and \eqref{eq:nd-Toda-finite-ev},
\begin{equation*}
  B^{(t)}\bm \phi^{(t)}(x_i)=x_i\bm \phi^{(t)}(x_i),\quad
  i=0, 1, \dots, N-1,
\end{equation*}
holds. This corresponds to the fact that the zeros of
the characteristic polynomials $\phi^{(t)}_N(x)$ are the eigenvalues
of $B^{(t)}$.
Moreover, it indicates that $\bm \phi^{(t)}(x_i)$ is
the eigenvector corresponding to the eigenvalue $x_i$.
From \eqref{eq:nd-Toda-finite-lax}, we have
\begin{align*}
  x\bm\phi^{(t+1)}(x)
  &=B^{(t+1)}\bm\phi^{(t+1)}(x)+\bm\phi^{(t+1)}_N(x)\\
  &=(L^{(t+1)}R^{(t+1)}+s^{(t+1)}I_N)\bm \phi^{(t+1)}(x)+\bm\phi^{(t+1)}_N(x)\\
  &=(R^{(t)}L^{(t)}+s^{(t)}I_N)\bm \phi^{(t+1)}(x)+\bm\phi^{(t+1)}_N(x).
\end{align*}
This yields the matrix form of the finite nd-Toda lattice \eqref{eq:nd-Toda-finite}:
\begin{equation*}
  B^{(t+1)}=L^{(t+1)}R^{(t+1)}+s^{(t+1)}I_N=R^{(t)}L^{(t)}+s^{(t)}I_N.
\end{equation*}
Since $L^{(t)}$ is always regular, it is shown that the tridiagonal matrices
$B^{(t)}$ and $B^{(t+1)}$ are similar:
\begin{equation*}
  B^{(t+1)}
  =R^{(t)}L^{(t)}+s^{(t)}I_N
  =\left(L^{(t)}\right)^{-1}\left(L^{(t)}R^{(t)}+s^{(t)}I_N\right)L^{(t)}
  =\left(L^{(t)}\right)^{-1}B^{(t)}L^{(t)}.
\end{equation*}
Therefore, the eigenvalues of $B^{(t)}$ are conserved under the time evolution.
This corresponds to the fact \eqref{eq:op-christoffel-N} that
the characteristic polynomial $\phi^{(t)}_N(x)$ is invariant under
the time evolution.
From the result, we can see that the spectral transformations
\eqref{eq:op-christoffel-matrix} and \eqref{eq:op-geronimus-matrix} correspond to
the $LU$ decomposition of the tridiagonal matrix $B^{(t)}$ with the shift $s^{(t)}$.

\subsection{The dqds algorithm}
\label{sec:dqds-algorithm}

For the finite nd-Toda lattice \eqref{eq:nd-Toda-finite},
let us introduce an auxiliary variable
\begin{equation}\label{eq:dqds-aux}
  d^{(t+1)}_n \coloneq q^{(t+1)}_n-e^{(t)}_{n+1},\quad
  n=0, 1, \dots, N-1.
\end{equation}
Then, equations \eqref{eq:nd-Toda-finite} are rewritten as
\begin{subequations}\label{eq:dqds}
  \begin{alignat}{2}
    &d^{(t+1)}_0=q^{(t)}_0-(s^{(t+1)}-s^{(t)}),\label{eq:dqds-d-0}\\
    &d^{(t+1)}_n=d^{(t+1)}_{n-1}\frac{q^{(t)}_n}{q^{(t+1)}_{n-1}}-(s^{(t+1)}-s^{(t)}),\label{eq:dqds-d-n}
    &\quad&n=1, 2, \dots, N-1,\\
    &q^{(t+1)}_n=e^{(t)}_{n+1}+d^{(t+1)}_n,
    &&n=0, 1, \dots, N-1,\\
    &e^{(t+1)}_n=e^{(t)}_n\frac{q^{(t)}_n}{q^{(t+1)}_{n-1}},
    &&n=1, 2, \dots, N-1,\\
    &e^{(t)}_0=e^{(t)}_N=0&& \text{for all $t \ge 0$}.
  \end{alignat}
\end{subequations}
These recurrence equations are called the \emph{dqds algorithm}.

The spectral transformations \eqref{eq:op-christoffel} and
\eqref{eq:op-geronimus} yields
\begin{equation*}
  \phi^{(t+1)}_n(x)=\frac{\phi^{(t-1)}_{n+1}(x)+d^{(t)}_n\phi^{(t)}_n(x)}{x-s^{(t)}},\quad
  n=0, 1, \dots, N-1.
\end{equation*}
Hence, we obtain
\begin{equation}\label{eq:dqds-d}
  d^{(t)}_n
  =-\frac{\phi^{(t-1)}_{n+1}(s^{(t)})}{\phi^{(t)}_n(s^{(t)})}
  =\frac{\tau^{(t)}_n\sigma^{(t)}_{n+1}}{\tau^{(t-1)}_{n+1}\tau^{(t+1)}_n},
\end{equation}
with
\begin{equation*}
  \sigma^{(t)}_n\coloneq|\mu^{(t-1)}_{i+j+1}-s^{(t)}\mu^{(t-1)}_{i+j}|_{0 \le i, j \le n-1}.
\end{equation*}
By a calculation similar to the derivation of the expanded form
\eqref{eq:nd-Toda-finite-hankel},
we obtain the expanded form of $\sigma^{(t)}_n$
\begin{equation}\label{eq:dqds-sigma}
  \sigma^{(t)}_n=\sum_{0 \le r_0<r_1<\dots<r_{n-1} \le N-1} \left(\prod_{i=0}^{n-1} \left(c^{(0)}_{r_i} (x_{r_i}-s^{(t)})\prod_{j=0}^{t-2} (x_{r_i}-s^{(j)})\right)\prod_{0 \le \nu_0<\nu_1 \le n-1} (x_{r_{\nu_1}}-x_{r_{\nu_0}})^2\right).
\end{equation}

We henceforth assume that, for the elements of the initial tridiagonal
matrix $B^{(0)}$, the following conditions are satisfied:
$u^{(0)}_0, u^{(0)}_1, \dots, u^{(0)}_{N-1}$ are all real and
$w^{(0)}_1, w^{(0)}_2, \dots, w^{(0)}_{N-1}$ are all real and positive.
Then, the tridiagonal matrix $B^{(0)}$ is similar to a real symmetric
tridiagonal matrix. The eigenvalues $x_0, x_1, \dots, x_{N-1}$ of $B^{(0)}$
are thus all real and simple.
In addition, the constants $c^{(0)}_0, c^{(0)}_1, \dots, c^{(0)}_{N-1}$ are all
real and positive by Theorem \ref{thm:gqf}.
Accordingly, the solution \eqref{eq:nd-Toda-finite-qe} and \eqref{eq:dqds-d} with
the expanded forms \eqref{eq:nd-Toda-finite-hankel} and
\eqref{eq:dqds-sigma} gives the next theorem.
\begin{theorem}
  Suppose that $u^{(0)}_0, u^{(0)}_1, \dots, u^{(0)}_{N-1}$ are all real, and
  $w^{(0)}_1, w^{(0)}_2, \dots, w^{(0)}_{N-1}$ are all real and positive.
  Choose the parameter $s^{(t)}$ as
  \begin{equation}\label{eq:dqds-s-cond}
    s^{(t)}<\min\{x_0, x_1, \dots, x_{N-1}\} \quad \text{for all $t \ge 0$}.
  \end{equation}
  Then, the variables $q^{(t)}_n$, $e^{(t)}_n$ and $d^{(t)}_n$ of the dqds
  algorithm \eqref{eq:dqds} are positive for all $n$ and $t \ge 0$.
\end{theorem}

This theorem guarantees that, under the assumption, the dqds algorithm
does not contain subtraction operations except the parameter terms
$-(s^{(t+1)}-s^{(t)})$ in equations \eqref{eq:dqds-d-0} and \eqref{eq:dqds-d-n}.
Namely, equations \eqref{eq:dqds} are the subtraction-free form of the
finite nd-Toda lattice \eqref{eq:nd-Toda-finite}.
It is known that this form improves the accuracy of the numerical computation.

The asymptotic analysis of the dqds algorithm proves convergence of the algorithm
and provides a method for accelerating the convergence.
The solution derived in Subsection~\ref{sec:op-finite-dimens-case}
is a key tool for the analysis.
Arrange the eigenvalues $x_0,  x_1, \dots, x_{N-1}$ of
the initial tridiagonal matrix $B^{(0)}$ in descending
order: $x_0>x_1>\dots>x_{N-1}$.
If the parameter $s^{(t)}$ chosen as \eqref{eq:dqds-s-cond}, namely
$s^{(t)}<x_{N-1}$, then the inequality $x_0-s^{(t)}>x_1-s^{(t)}>\dots>x_{N-1}-s^{(t)}>0$ holds for all $t \ge 0$.
Under this assumption, by the solution \eqref{eq:nd-Toda-finite-qe} with
the expanded form \eqref{eq:nd-Toda-finite-hankel},
we obtain the asymptotic behaviour for $t \gg 0$:
\begin{equation*}
  q^{(t)}_n=x_n-s^{(t)}+O\left(\max\left\{\frac{\prod_{j=0}^{t}(x_n-s^{(j)})}{\prod_{j=0}^{t-1} (x_{n-1}-s^{(j)})}, \frac{\prod_{j=0}^{t}(x_{n+1}-s^{(j)})}{\prod_{j=0}^{t-1} (x_{n}-s^{(j)})}\right\}\right),\quad
  e^{(t)}_n=O\left(\frac{\prod_{j=0}^{t-1} (x_n-s^{(j)})}{\prod_{j=0}^{t}(x_{n-1}-s^{(j)})}\right).
\end{equation*}
This shows that $q^{(t)}_n$ and $e^{(t)}_n$ converge to $x_n-s^{(t)}$
and $0$ as $t \to +\infty$, respectively.
Hence, it is shown that the dqds algorithm \eqref{eq:dqds}
with appropriate parameters $s^{(t)}$ computes the eigenvalues of
a given real symmetric tridiagonal matrix.
It is clear that
the convergence speed depends on $\frac{x_n-s^{(t)}}{x_{n-1}-s^{(t)}}$.
Therefore, we should choose the parameter $s^{(t)}<x_{N-1}$
as close as possible to the minimum eigenvalue $x_{N-1}$ for fast computation.
The acceleration parameter $s^{(t)}$ is called the \emph{origin shift}.

\section{\Rii polynomials, \Rii chain, and generalized eigenvalue algorithm}
\label{sec:rii-polynomials-rii}

We shall extend the discussion in Section~\ref{sec:orth-polyn-nd}
for tridiagonal matrix pencils and
its associated nonautonomous discrete integrable system.

\subsection{Infinite dimensional case}
\label{sec:infin-dimens-case-rii}
Let us consider two tridiagonal semi-infinite matrices in the following forms:
\begin{alignat*}{2}
  &A^{(t)}=
  \begin{pmatrix}
    v^{(t)}_0 & \kappa_t\\
    \lambda_1 w^{(t)}_1 & v^{(t)}_1 & \kappa_{t+1}\\
    & \lambda_2 w^{(t)}_2 & v^{(t)}_2 & \kappa_{t+2}\\
    && \lambda_3 w^{(t)}_3 & \ddots & \ddots\\
    &&& \ddots & \ddots
  \end{pmatrix},&\quad&
  v^{(t)}_n, \kappa_{t+n}, \lambda_n \in \mathbb C,\quad
  w^{(t)}_n \in \mathbb C-\{0\},\\
  &B^{(t)}=
  \begin{pmatrix}
    u^{(t)}_0 & 1\\
    w^{(t)}_1 & u^{(t)}_1 & 1\\
    & w^{(t)}_2 & u^{(t)}_2 & 1\\
    && w^{(t)}_3 & \ddots & \ddots\\
    &&& \ddots & \ddots
  \end{pmatrix},&&
  u^{(t)}_n\in \mathbb C.
\end{alignat*}
Let $A^{(t)}_n$ and $B^{(t)}_n$ denote the $n$-th order leading principal
submatrices of $A^{(t)}$ and $B^{(t)}$, respectively.
We now define a polynomial sequence $\{\varphi^{(t)}_n(x)\}_{n=0}^\infty$ by
\begin{equation*}
  \varphi^{(t)}_0(x)\coloneq 1,\quad
  \varphi^{(t)}_n(x)\coloneq \det(xB^{(t)}_n-A^{(t)}_n),\quad
  n=1, 2, 3, \dots.
\end{equation*}
The polynomial $\varphi^{(t)}_n(x)$ is a monic polynomial of degree $n$.
In the same manner as in the case of monic orthogonal polynomials
in Section~\ref{sec:orth-polyn-nd},
we obtain the three-term recurrence relation
\begin{equation}\label{eq:rii-trr}
  \varphi^{(t)}_{n+1}(x)=(u^{(t)}_n x-v^{(t)}_n)\varphi^{(t)}_n(x)-w^{(t)}_n(x-\kappa_{t+n-1})(x-\lambda_n)\varphi^{(t)}_{n-1}(x),\quad
  n=0, 1, 2, \dots,
\end{equation}
where we set $w^{(t)}_0\coloneq 0$ and $\varphi^{(t)}_{-1}(x)\coloneq 0$.
We will assume in what follows that
all the parameters $\kappa_{t+k}$ and $\lambda_k$, $k=0, 1, 2, \dots$,
are not zeros of the polynomial $\varphi^{(t)}_n(x)$ for all $n \in \mathbb N$.
The polynomials $\{\varphi^{(t)}_n(x)\}_{n=0}^\infty$ are called
the \emph{\Rii polynomials} with respect to $\mathcal L^{(t)}$,
introduced by Ismail and Masson \cite{ismail1995goa}.

We introduce the notations
\begin{equation*}
  K^{(t)}_k(x)\coloneq \prod_{j=0}^{k-1} (x-\kappa_{t+j}),\quad
  K_k(x)\coloneq K^{(0)}_k(x)=\prod_{j=0}^{k-1} (x-\kappa_j),\quad
  L_l(x)\coloneq \prod_{j=1}^{l} (x-\lambda_j),
\end{equation*}
and $\mathcal D(\mathcal L^{(t)})$ a linear space spanned by
the rational functions $\frac{x^m}{K^{(t)}_k(x)L_l(x)}$, $k, l=0, 1, 2, \dots$;
$m=0, 1, \dots, k+l$.
The following Favard type theorem is proved.

\begin{theorem}[Favard type theorem for the \Rii polynomials \cite{ismail1995goa}]\label{thm:rii-favard}
  For the \Rii polynomials $\{\varphi^{(t)}_n(x)\}_{n=0}^\infty$ and
  any nonzero constants $h^{(t)}_0$ and $h^{(t)}_1$, which satisfy $h^{(t)}_0 \ne h^{(t)}_1$,
  there exists a unique linear functional defined on
  $\mathcal D(\mathcal L^{(t)})$ such that the orthogonality relation
  \begin{equation*}\label{eq:rii-orthogonality}
    \mathcal L^{(t)}\left[\frac{x^m\varphi^{(t)}_n(x)}{K^{(t)}_n(x)L_n(x)}\right]
    =h^{(t)}_n \delta_{m, n},\quad
    n=0, 1, 2, \dots,\quad
    m=0, 1, 2, \dots, n,
  \end{equation*}
  holds, where $h^{(t)}_n$, $n=2, 3, \dots$, are some nonzero constants.
\end{theorem}

In the rest of this paper, we consider the \emph{monic} \Rii polynomials,
i.e., the case where $u^{(t)}_n=1+w^{(t)}_n$ holds for all $n=0, 1, 2, \dots$.
For general tridiagonal semi-infinite matrices of the form $B^{(t)}$,
if $\det(B^{(t)}_n) \ne 0$ holds for all $n=1, 2, 3, \dots$,
then $\left\{\frac{\varphi^{(t)}_n(x)}{\det(B^{(t)}_n)}\right\}_{n=0}^\infty$ are
the monic \Rii polynomials.
Therefore, the following argument is valid for such matrices.

The moment of the \Rii linear functional $\mathcal L^{(t)}$ is introduced by
\begin{equation}\label{eq:rii-moment}
  \mu^{k, l, t}_m\coloneq \mathcal L^{(t)}\left[\frac{x^m}{K^{(t)}_k(x) L_l(x)}\right],\quad
  k, l=0, 1, 2, \dots, \quad m=0, 1, \dots, k+l,
\end{equation}
and its Hankel determinant by
\begin{equation*}
  \tau^{k, l, t}_0\coloneq 1,\quad
  \tau^{k, l, t}_n
  \coloneq |\mu^{k, l, t}_{i+j}|_{0 \le i, j \le n-1}
  =
  \begin{vmatrix}
    \mu^{k, l, t}_0 & \mu^{k, l, t}_1 & \dots & \mu^{k, l, t}_{n-1}\\
    \mu^{k, l, t}_1 & \mu^{k, l, t}_2 & \dots & \mu^{k, l, t}_{n}\\
    \vdots & \vdots & & \vdots\\
    \mu^{k, l, t}_{n-1} & \mu^{k, l, t}_n & \dots & \mu^{k, l, t}_{2n-2}
  \end{vmatrix},\quad
  n=1, 2, 3, \dots.
\end{equation*}
Then, the determinant expression of the monic \Rii polynomials
$\{\varphi^{(t)}_n(x)\}_{n=0}^\infty$ is presented:
\begin{equation}\label{eq:rii-determinant}
  \varphi^{(t)}_n(x)=\frac{1}{\tau^{n, n, t}_n}
  \begin{vmatrix}
    \mu^{n, n, t}_0 & \mu^{n, n, t}_1 & \dots & \mu^{n, n, t}_{n-1} & \mu^{n, n, t}_n\\
    \mu^{n, n, t}_1 & \mu^{n, n, t}_2 & \dots & \mu^{n, n, t}_{n} & \mu^{n, n, t}_{n+1}\\
    \vdots & \vdots & & \vdots & \vdots\\
    \mu^{n, n, t}_{n-1} & \mu^{n, n, t}_n & \dots & \mu^{n, n, t}_{2n-2} & \mu^{n, n, t}_{2n-1}\\
    1 & x & \dots & x^{n-1} & x^n
  \end{vmatrix},\quad n=0, 1, 2, \dots.
\end{equation}

The discrete time evolution for the monic \Rii polynomials is introduced by
an analogue of the spectral transformations for monic orthogonal polynomials:
\begin{subequations}\label{eq:rii-spectral-transform}
  \begin{gather}
    (x-s^{(t)})(1+q^{(t)}_n)\varphi^{(t+1)}_n(x)=\varphi^{(t)}_{n+1}(x)+q^{(t)}_n(x-\kappa_{t+n})\varphi^{(t)}_n(x),\label{eq:rii-christoffel}\\
    (1+e^{(t)}_n)\varphi^{(t)}_n(x)=\varphi^{(t+1)}_n(x)+e^{(t)}_n(x-\lambda_n)\varphi^{(t+1)}_{n-1}(x)\label{eq:rii-geronimus}
  \end{gather}
\end{subequations}
for $n=0, 1, 2, \dots$, where
\begin{equation}\label{eq:rii-q}
  q^{(t)}_n \coloneq -(s^{(t)}-\kappa_{t+n})\frac{\varphi^{(t)}_{n+1}(s^{(t)})}{\varphi^{(t)}_n(s^{(t)})},\quad
  n=0, 1, 2, \dots,
\end{equation}
and $e^{(t)}_n$ is the variable determined by
the compatibility condition:
\begin{subequations}\label{eq:rii}
  \begin{align}
    1+w^{(t+1)}_n
    &=-q^{(t+1)}_n-e^{(t+1)}_n\frac{1+q^{(t+1)}_n}{1+q^{(t+1)}_{n-1}}+(1+q^{(t+1)}_n)(1+e^{(t+1)}_n)\nonumber\\
    &=-q^{(t)}_n\frac{1+e^{(t)}_{n+1}}{1+e^{(t)}_n}-e^{(t)}_{n+1}+(1+q^{(t)}_n)(1+e^{(t)}_{n+1}),\label{eq:rii-1+w}\\
    v^{(t+1)}_n
    &=-\kappa_{t+n+1} q^{(t+1)}_n-\lambda_n e^{(t+1)}_n\frac{1+q^{(t+1)}_n}{1+q^{(t+1)}_{n-1}}+s^{(t+1)}(1+q^{(t+1)}_n)(1+e^{(t+1)}_n)\nonumber\\
    &=-\kappa_{t+n} q^{(t)}_n\frac{1+e^{(t)}_{n+1}}{1+e^{(t)}_n}-\lambda_{n+1}e^{(t)}_{n+1}+s^{(t)}(1+q^{(t)}_n)(1+e^{(t)}_{n+1}),\label{eq:rii-v}\\
    w^{(t+1)}_n
    &=q^{(t+1)}_{n-1}e^{(t+1)}_n\frac{1+q^{(t+1)}_n}{1+q^{(t+1)}_{n-1}}
    =q^{(t)}_n e^{(t)}_n\frac{1+e^{(t)}_{n+1}}{1+e^{(t)}_{n}}\label{eq:rii-w}
  \end{align}
  with the boundary condition
  \begin{equation}
    e^{(t)}_0=0 \quad \text{for all $t \ge 0$}.
  \end{equation}
\end{subequations}
This is the semi-infinite \emph{monic type \Rii chain}.
Note that, since \eqref{eq:rii-1+w} and \eqref{eq:rii-w} are identical,
there are the two independent equations that determine the time evolution of the
two variables $q^{(t)}_n$ and $e^{(t)}_n$.
It is readily verified that if $\{\varphi^{(t)}_{n}(x)\}_{n=0}^\infty$ are
the monic \Rii polynomials with respect to $\mathcal L^{(t)}$,
then polynomials $\{\varphi^{(t+1)}_n(x)\}_{n=0}^\infty$ defined by
the spectral transformation \eqref{eq:rii-christoffel} are again
the monic \Rii polynomials, where the corresponding \Rii linear functional is
defined by
\begin{equation}\label{eq:rii-lf-evolution}
  \mathcal L^{(t+1)}[R(x)]\coloneq \mathcal L^{(t)}\left[\frac{x-s^{(t)}}{x-\kappa_t} R(x)\right]
\end{equation}
for all $R(x) \in \mathcal D(\mathcal L^{(t+1)})$.

Let us derive a solution to the monic type \Rii chain.
By the definition of the moment \eqref{eq:rii-moment} and the time evolution of
the linear functional \eqref{eq:rii-lf-evolution}, we obtain the relations
\begin{subequations}\label{eq:rii-dispersion}
  \begin{gather}
    \mu^{k, l, t}_{m}
    =\mu^{k+1, l, t}_{m+1}-\kappa_{t+k}\mu^{k+1, l, t}_{m}
    =\mu^{k, l+1, t}_{m+1}-\lambda_{l+1}\mu^{k, l+1, t}_{m},\label{eq:rii-moment-kappa-lambda}\\
    \mu^{k, l, t+1}_m
    =\mu^{k+1, l, t}_{m+1}-s^{(t)}\mu^{k+1, l, t}_m.\label{eq:rii-moment-time-evol}
  \end{gather}
\end{subequations}
The relation \eqref{eq:rii-moment-time-evol}, the determinant expression
of the monic \Rii polynomials \eqref{eq:rii-determinant}, and the definition of
the variable \eqref{eq:rii-q} lead to
\begin{equation}\label{eq:rii-sol-q}
  q^{(t)}_n=(s^{(t)}-\kappa_{t+n})^{-1}\frac{\tau^{n, n, t}_n\tau^{n, n+1, t+1}_{n+1}}{\tau^{n-1, n, t+1}_n\tau^{n+1, n+1, t}_{n+1}}.
\end{equation}
Next, the relation \eqref{eq:rii-moment-kappa-lambda} and
the spectral transformation \eqref{eq:rii-christoffel} yield
\begin{equation*}
  1+q^{(t)}_n=(\kappa_{t+n}-s^{(t)})^{-1}\frac{\varphi^{(t)}_{n+1}(\kappa_{t+n})}{\varphi^{(t+1)}_n(\kappa_{t+n})}
  =(s^{(t)}-\kappa_{t+n})^{-1}\frac{\tau^{n, n+1, t}_{n+1}\tau^{n, n, t+1}_n}{\tau^{n+1, n+1, t}_{n+1}\tau^{n-1, n, t+1}_n}.
\end{equation*}
Further,
the Jacobi identity for determinants \cite[Section~2.6]{hirota2004dms} proves
the bilinear equation
\begin{equation*}
  \tau^{k-1, l-1, t}_n\tau^{k, l, t}_n-\tau^{k-1, l, t}_n\tau^{k, l-1, t}_n-\tau^{k-1, l-1, t}_{n-1}\tau^{k, l, t}_{n+1}=0.
\end{equation*}
By using this bilinear equation and the three-term recurrence relation
\eqref{eq:rii-trr}, we obtain
\begin{equation*}
  w^{(t)}_n
  =\frac{\mathcal L^{(t)}\left[\frac{x^{n+1}\varphi^{(t)}_{n+1}(x)}{K^{(t)}_{n+1}(x)L_{n+1}(x)}-\frac{x^{n}\varphi^{(t)}_{n}(x)}{K^{(t)}_{n}(x)L_{n}(x)}\right]}{\mathcal L^{(t)}\left[\frac{x^n\varphi^{(t)}_n(x)}{K^{(t)}_{n}(x)L_{n}(x)}-\frac{x^{n-1}\varphi^{(t)}_{n-1}(x)}{K^{(t)}_{n-1}(x)L_{n-1}(x)}\right]}
  =\frac{\tau^{n-1, n-1, t}_{n-1}\tau^{n, n+1, t}_{n+1}\tau^{n+1, n, t}_{n+1}}{\tau^{n-1, n, t}_n\tau^{n, n-1, t}_n\tau^{n+1, n+1, t}_{n+1}}.
\end{equation*}
Hence, from equation \eqref{eq:rii-w} and these formulae, we find a solution
\begin{equation}\label{eq:rii-sol-e}
  e^{(t)}_n
  =\frac{w^{(t)}_n}{q^{(t)}_{n-1}}\frac{1+q^{(t)}_{n-1}}{1+q^{(t)}_n}
  =(s^{(t)}-\kappa_{t+n})\frac{\tau^{n-1, n-1, t+1}_{n-1}\tau^{n+1, n, t}_{n+1}}{\tau^{n, n-1, t}_n\tau^{n, n, t+1}_n}.
\end{equation}
If the moments $\mu^{k, l, t}_m$ are arbitrary functions satisfying
the relations \eqref{eq:rii-dispersion}, e.g.,
\begin{equation*}
  \mu^{k, l, t}_m=\int_\Omega \frac{x^m \prod_{j=0}^{t-1}(x-s^{(j)})}{K_{t+n}(x)L_n(x)} \omega(x)\, \mathrm{d}x,
\end{equation*}
then \eqref{eq:rii-sol-q} and \eqref{eq:rii-sol-e} give a solution to
the monic type \Rii chain \eqref{eq:rii} expressed
by the Hankel determinant $\tau^{k, l, t}_n$.

The reason why the Hankel determinant appears can be explained
from the point of view of the discrete two-dimensional
Toda hierarchy \cite{tsujimoto2010dso}.
Note that there is another determinant expression of the \Rii polynomials and
a solution to the \Rii chain: the Casorati-type determinant solution
\cite{spiridonov2003tbr,mukaihira2006dsn}.

\subsection{Finite dimensional case}
\label{sec:finite-dimens-case}

In this subsection, we will derive the solution to the initial value problem
and the convergence theorem for the monic type finite \Rii chain.

Let us start with a pair of tridiagonal matrices of order $N$:
\begin{equation}\label{eq:rii-tridiagonal-finite}
  \hspace{-\mathindent}
  A^{(t)}=
  \begin{pmatrix}
    v^{(t)}_0 & \kappa_t\\
    \lambda_1 w^{(t)}_1 & v^{(t)}_1 & \kappa_{t+1}\\
    & \lambda_2 w^{(t)}_2 & \ddots & \ddots\\
    && \ddots & \ddots & \kappa_{t+N-2}\\
    &&& \lambda_{N-1}w^{(t)}_{N-1} & v^{(t)}_{N-1}
  \end{pmatrix},\
  B^{(t)}=
  \begin{pmatrix}
    1 & 1\\
    w^{(t)}_1 & 1+w^{(t)}_1 & 1\\
    & w^{(t)}_2 & \ddots & \ddots\\
    && \ddots & \ddots & 1\\
    &&& w^{(t)}_{N-1} & 1+w^{(t)}_{N-1}
  \end{pmatrix}.
\end{equation}
The corresponding monic type finite \Rii chain is
\begin{subequations}\label{eq:rii-finite}
  \begin{gather}
    \begin{multlined}[b]
      \kappa_{t+n+1} q^{(t+1)}_n+\lambda_n e^{(t+1)}_n\frac{1+q^{(t+1)}_n}{1+q^{(t+1)}_{n-1}}-s^{(t+1)}(1+q^{(t+1)}_n)(1+e^{(t+1)}_n)\\
      =\kappa_{t+n} q^{(t)}_n\frac{1+e^{(t)}_{n+1}}{1+e^{(t)}_n}+\lambda_{n+1}e^{(t)}_{n+1}-s^{(t)}(1+q^{(t)}_n)(1+e^{(t)}_{n+1}),
    \end{multlined}\\
    q^{(t+1)}_{n-1}e^{(t+1)}_n\frac{1+q^{(t+1)}_n}{1+q^{(t+1)}_{n-1}}
    =q^{(t)}_n e^{(t)}_n\frac{1+e^{(t)}_{n+1}}{1+e^{(t)}_{n}},\\
    e^{(t)}_0=e^{(t)}_N=0 \quad \text{for all $t \ge 0$}.
  \end{gather}
\end{subequations}
To derive the solution to the initial value problem for the
monic type finite \Rii chain \eqref{eq:rii-finite}, we consider the monic finite
\Rii polynomials $\{\varphi^{(t)}_n(x)\}_{n=0}^\infty$ defined by
$\varphi^{(t)}_n(x)\coloneq \det(xB^{(t)}_n-A^{(t)}_n)$.
We should remark that $\varphi^{(t)}_N(x)$ is the characteristic polynomial
of the tridiagonal matrix pencil $(A^{(t)}, B^{(t)})$; the zeros of
the polynomial $\varphi^{(t)}_N(x)$ are the generalized eigenvalues of
the matrix pencil $(A^{(t)}, B^{(t)})$, i.e., the solutions of the equation
\begin{equation*}
  A^{(t)}\bm\Phi=xB^{(t)}\bm\Phi,\quad
  x \in \mathbb C,\quad \bm\Phi \in \mathbb C^N -\{\bm 0\}.
\end{equation*}

Let $\mathcal D(\mathcal L^{(t)})$ be a linear space spanned by
the rational functions $\frac{x^m}{K^{(t)}_N(x)L_N(x)}$,
$m=0, 1, 2, \dots$.
For the monic finite \Rii polynomials $\{\varphi^{(t)}_n(x)\}_{n=0}^N$
and any nonzero constant $H^{(t)}$,
there exists a unique linear functional defined on
$\mathcal D(\mathcal L^{(t)})$ such that the orthogonality relation
\begin{subequations}
  \begin{alignat}{3}
    &\mathcal L^{(t)}\left[\frac{x^m \varphi^{(t)}_n(x)}{K^{(t)}_n(x)L_n(x)}\right]=h^{(t)}_n \delta_{m, n},&\quad&
    n=0, 1, \dots, N-1,&\quad&
    m=0, 1, \dots, n,\\
    \intertext{and the terminating condition}
    &\mathcal L^{(t)}\left[\frac{x^m \varphi^{(t)}_N(x)}{K^{(t)}_{k}(x)L_{l}(x)}\right]=0,&&
    k, l=0, 1, \dots, N,&&
    m=0, 1, 2, \dots,\label{eq:rii-finite-orthogonality-2}
  \end{alignat}
\end{subequations}
hold,
where the constants $h^{(t)}_0, h^{(t)}_1, \dots, h^{(t)}_{N-1}$ are given
by solving the following linear equation
\begin{equation*}
  \begin{pmatrix}
    -1 & 1\\
    w^{(t)}_1 & -(1+w^{(t)}_1) & 1\\
    & w^{(t)}_2 & \ddots & \ddots\\
    && \ddots & \ddots & 1\\
    &&& w^{(t)}_{N-1} & -(1+w^{(t)}_{N-1})
  \end{pmatrix}
  \begin{pmatrix}
    h^{(t)}_0\\
    h^{(t)}_1\\
    \vdots\\
    h^{(t)}_{N-1}
  \end{pmatrix}=
  \begin{pmatrix}
    -H^{(t)}\\
    0\\
    \vdots\\
    0
  \end{pmatrix},
\end{equation*}
i.e.,
\begin{gather*}
  h^{(t)}_0=H^{(t)}(1+w^{(t)}_1+w^{(t)}_1w^{(t)}_2+\dots+w^{(t)}_1 w^{(t)}_2\dots w^{(t)}_{N-1}),\\
  h^{(t)}_1=H^{(t)}(w^{(t)}_1+w^{(t)}_1w^{(t)}_2+\dots+w^{(t)}_1 w^{(t)}_2\dots w^{(t)}_{N-1}),\\
  \vdots\\
  h^{(t)}_{N-1}=H^{(t)}w^{(t)}_1 w^{(t)}_2\dots w^{(t)}_{N-1}.
\end{gather*}
Note that, for the infinite dimensional case,
there are two degrees of freedom: the choice of
the two constants $h^{(t)}_0$ and $h^{(t)}_1$ (see Theorem~\ref{thm:rii-favard}).
For the finite dimensional case, however, there is only one degree of freedom:
the choice of the constant $H^{(t)}$.
The cause of this is the terminating condition
\eqref{eq:rii-finite-orthogonality-2}.

To derive a realization of $\mathcal L^{(t)}$,
we give a quadrature formula for the \Rii linear functional.
Suppose that all the zeros $x_0, x_1, \dots, x_{N-1}$ of
the characteristic polynomial $\varphi^{(t)}_{N}(x)$ are simple.
\begin{theorem}[The quadrature formula for the \Rii linear functional]
  Let $x_0, x_1, \dots, x_{N-1}$ be the simple zeros of
  the characteristic polynomial $\varphi^{(t)}_{N}(x)$.
  For the linear functional $\mathcal L^{(t)}$ of the monic finite
  \Rii polynomials $\{\varphi^{(t)}_n(x)\}_{n=0}^N$, there exist some constants
  $c^{(t)}_0, c^{(t)}_1, \dots, c^{(t)}_{N-1}$ such that
  \begin{equation}\label{eq:rii-quadrature}
    \mathcal L^{(t)}[R(x)]=\sum_{i=0}^{N-1} c^{(t)}_i R(x_i)
  \end{equation}
  holds for all $R(x) \in \mathcal D(\mathcal L^{(t)})$.
\end{theorem}
\begin{proof}
  This proof is an analogue of the proof to the Gauss quadrature formula
  (Theorem \ref{thm:gqf}).
  For the given rational function $R(x)$,
  consider the following interpolation rational function
  \begin{equation*}
    \Lambda(x)\coloneq \sum_{i=0}^{N-1} \ell^{(t)}_i(x)R(x_i),
  \end{equation*}
  where
  \begin{equation*}
    \ell^{(t)}_i(x)\coloneq \frac{\varphi^{(t)}_N(x) K^{(t)}_{N}(x_i) L_{N}(x_i)}{(x-x_i) K^{(t)}_{N}(x)L_{N}(x){\varphi'}^{(t)}_N(x_i)},\quad
    i=0, 1, \dots, N-1.
  \end{equation*}
  It is readily shown that
  \begin{equation*}
    \ell^{(t)}_i(x_j)=\delta_{i, j},\quad
    i, j=0, 1, \dots, N-1,
  \end{equation*}
  holds. Let
  \begin{equation*}
    Q(x) \coloneq R(x)-\Lambda(x).
  \end{equation*}
  Then, the numerator of $Q(x)$ is a polynomial
  that has zeros at $x_0, x_1, \dots, x_{N-1}$.
  Since $R(x) \in \mathcal D(\mathcal L^{(t)})$,
  there exists a polynomial $P(x)$ such that
  \begin{equation*}
    Q(x)=\frac{P(x)\varphi^{(t)}_N(x)}{K^{(t)}_{N}(x)L_{N}(x)}.
  \end{equation*}
  By the terminating condition \eqref{eq:rii-finite-orthogonality-2},
  we obtain
  \begin{align*}
    \mathcal L^{(t)}[R(x)]
    &=\mathcal L^{(t)}[\Lambda(x)]+\mathcal L^{(t)}[Q(x)]\\
    &=\sum_{i=0}^{N-1}\mathcal L^{(t)}[\ell^{(t)}_i(x)]R(x_i)+\mathcal L^{(t)}\left[\frac{P(x)\varphi^{(t)}_N(x)}{K^{(t)}_{N}(x)L_{N}(x)}\right]\\
    &=\sum_{i=0}^{N-1}\mathcal L^{(t)}[\ell^{(t)}_i(x)]R(x_i).
  \end{align*}
  Set $c^{(t)}_i\coloneq \mathcal L^{(t)}[\ell^{(t)}_i(x)]$, $i=0, 1, \dots, N-1$,
  then the proof is completed.
\end{proof}

Zhedanov \cite{zhedanov1999brf} derived a formula to calculate the constants $c^{(t)}_0, c^{(t)}_1, \dots, c^{(t)}_{N-1}$.
He used the second kind polynomials to derive it.
Here, we give a direct calculation to check his result.
From the quadrature formula \eqref{eq:rii-quadrature},
the moment is written as
\begin{equation*}
  \mu^{k, l, t}_m
  =\mathcal L^{(t)}\left[\frac{x^m}{K^{(t)}_k(x) L_l(x)}\right]
  =\sum_{i=0}^{N-1} \frac{c^{(t)}_i x_i^m}{K^{(t)}_k(x_i) L_l(x_i)}.
\end{equation*}
In the same manner as in Subsection~\ref{sec:op-finite-dimens-case},
we thus obtain the following formulae for $j=0, 1, \dots, N-1$:
\begin{gather*}
  \varphi^{(t)}_{N-1}(x_j)=\frac{1}{\tau^{N-1, N-1, t}_{N-1}}\prod_{\substack{i=0\\i \ne j}}^{N-1}\frac{c^{(t)}_i (x_j-x_i)}{K^{(t)}_{N-1}(x_i)L_{N-1}(x_i)}\prod_{\substack{0\le \nu_0<\nu_1\le N-1\\\nu_0 \ne j, \nu_1 \ne j}}(x_{\nu_1}-x_{\nu_0})^2,\\
         {\varphi'}^{(t)}_N(x_j)=\prod_{\substack{i=0\\ i\ne j}}^{N-1}(x_j-x_i),
\end{gather*}
and
\begin{gather*}
  \tau^{N-1, N-1, t}_N=\prod_{i=0}^{N-1}\frac{c^{(t)}_i}{K^{(t)}_{N-1}(x_i)L_{N-1}(x_i)}\prod_{\substack{0\le \nu_0<\nu_1\le N-1\\\nu_0 \ne j, \nu_1 \ne j}}(x_{\nu_1}-x_{\nu_0})^2,\\
  h^{(t)}_{N-1}=\frac{\tau^{N-1, N-1, t}_N}{\tau^{N-1, N-1, t}_{N-1}}.
\end{gather*}
Hence, we find the formula
\begin{equation*}
  c^{(t)}_i=\frac{h^{(t)}_{N-1} K^{(t)}_{N-1}(x_i)L_{N-1}(x_i)}{\varphi^{(t)}_{N-1}(x_i){\varphi'}^{(t)}_N(x_i)},\quad
  i=0, 1, \dots, N-1.
\end{equation*}

For the finite dimensional case, in the same manner as for the
monic finite orthogonal polynomials (see Subsection~\ref{sec:op-finite-dimens-case}),
the characteristic polynomial is invariant under the time evolution:
\begin{equation*}
  \varphi^{(t+1)}_N(x)=\varphi^{(t)}_N(x).
\end{equation*}
From the results in Subsection~\ref{sec:infin-dimens-case-rii},
we can thus see that
the solution to the initial value problem for
the monic type finite \Rii chain is given by
\begin{equation*}\label{eq:rii-sol-qe}
  q^{(t)}_n
  =(s^{(t)}-\kappa_{t+n})^{-1}\frac{\tau^{n, n, t}_n\tau^{n, n+1, t+1}_{n+1}}{\tau^{n-1, n, t+1}_n\tau^{n+1, n+1, t}_{n+1}},\quad
  e^{(t)}_n
  =(s^{(t)}-\kappa_{t+n})\frac{\tau^{n-1, n-1, t+1}_{n-1}\tau^{n+1, n, t}_{n+1}}{\tau^{n, n-1, t}_n\tau^{n, n, t+1}_n},
\end{equation*}
where, because the moment is concretely given by
\begin{equation*}
  \mu^{k, l, t}_m=\sum_{i=0}^{N-1}\frac{c^{(0)}_i x_i^m \prod_{j=0}^{t-1}(x_i-s^{(j)})}{K_{t+k}(x_i)L_l(x_i)},
\end{equation*}
the expanded form of the Hankel determinant is
\begin{equation*}
  \tau^{k, l, t}_n=\sum_{0 \le r_0<r_1<\dots<r_{n-1}\le N-1}\left(\prod_{i=0}^{n-1} \frac{c^{(0)}_{r_i}\prod_{j=0}^{t-1}(x_{r_i}-s^{(j)})}{K_{t+k}(x_{r_i})L_l(x_{r_i})}\prod_{0\le \nu_0<\nu_1\le n-1}(x_{r_{\nu_1}}-x_{r_{\nu_0}})^2\right).
\end{equation*}

The solution derived above yields the following theorem.

\begin{theorem}[Convergence theorem for the monic type finite \Rii chain]\label{thm:rii-converge}
  Suppose that all the generalized eigenvalues $x_0, x_1, \dots, x_{N-1}$
  of the initial tridiagonal matrix pencil $(A^{(0)}, B^{(0)})$ are
  real, simple and arranged
  in descending order as $x_0>x_1>\dots>x_{N-1}$.
  Choose the parameters $s^{(t)}$ and $\kappa_{t+N-1}$ as $x_{N-1}>s^{(t)}$ and
  $x_{N-1}\gg \kappa_{t+N-1}$ for all $t \ge 0$, respectively.
  Then, we have the asymptotics of the variables for $t \gg 0$:
  \begin{gather*}
    \begin{split}
      q^{(t)}_n&=\frac{x_n-s^{(t)}}{s^{(t)}-\kappa_{t+n}}\\
      &\qquad +O\left(\max\left\{
      \frac{\prod_{j=0}^t(x_n-s^{(j)})}{\prod_{j=0}^{t-1}(x_{n-1}-s^{(j)})}\frac{\prod_{j=0}^{t+n-1}(x_{n-1}-\kappa_{j})}{\prod_{j=0}^{t+n-1}(x_n-\kappa_j)},
      \frac{\prod_{j=0}^t(x_{n+1}-s^{(j)})}{\prod_{j=0}^{t-1}(x_{n}-s^{(j)})}\frac{\prod_{j=0}^{t+n-1}(x_{n}-\kappa_{j})}{\prod_{j=0}^{t+n-1}(x_{n+1}-\kappa_j)}
      \right\}\right),
    \end{split}\\
    e^{(t)}_n=O\left(\frac{\prod_{j=0}^{t-1}(x_n-s^{(j)})}{\prod_{j=0}^{t}(x_{n-1}-s^{(j)})}\frac{\prod_{j=0}^{t+n-1}(x_{n-1}-\kappa_{j})}{\prod_{j=0}^{t+n}(x_n-\kappa_j)}\right).
  \end{gather*}
  Hence, the variables $q^{(t)}_n$ and $e^{(t)}_n$ converge to $\frac{x_n-s^{(t)}}{s^{(t)}-\kappa_{t+n}}$ and $0$ as $t \to +\infty$, respectively.
\end{theorem}
This theorem implies that, from \eqref{eq:rii},
the elements $v^{(t)}_n$ and $w^{(t)}_n$ of the tridiagonal matrices $A^{(t)}$
and $B^{(t)}$ converge to $x_n$ and $0$ as $t \to +\infty$, respectively.
Further, we can see that the parameters $s^{(t)}$ and $\kappa_{t+n}$ determine
the convergence speed;
the parameter $s^{(t)}$ works as the origin shift,
which is the same as for the dqds algorithm
(see the end of Subsection \ref{sec:dqds-algorithm}).

Next,
we discuss the matrix form of the monic type finite \Rii chain.
Introduce the rational functions defined by the following three-term
recurrence relation:
\begin{gather}
  \Phi^{(t)}_{-1}(x)\coloneq 0,\quad
  \Phi^{(t)}_0(x)\coloneq 1,\nonumber\\
  (x-\kappa_{t+n})\Phi^{(t)}_{n+1}(x)\coloneq
  -\left((1+w^{(t)}_n)x-v^{(t)}_n\right)\Phi^{(t)}_n(x)-w^{(t)}_n(x-\lambda_n)\Phi^{(t)}_{n-1}(x),\quad
  n=0, 1, \dots, N-1.\label{eq:rii-trr-rational}
\end{gather}
By comparing to the three-term recurrence relation \eqref{eq:rii-trr},
the relation
\begin{equation*}
  \Phi^{(t)}_n(x)=\frac{\varphi^{(t)}_n(x)}{K^{(t)}_n(x)},\quad
  n=0, 1, \dots, N,
\end{equation*}
is verified.
Let
\begin{equation*}
  \bm\Phi^{(t)}(x)\coloneq
  \begin{pmatrix}
    \Phi^{(t)}_0(x)\\
    \Phi^{(t)}_1(x)\\
    \vdots\\
    \Phi^{(t)}_{N-1}(x)
  \end{pmatrix},\quad
  \bm\Phi^{(t)}_N(x)\coloneq
  \begin{pmatrix}
    0\\
    \vdots\\
    0\\
    \Phi^{(t)}_{N}(x)
  \end{pmatrix}.
\end{equation*}
Then, the three-term recurrence relation \eqref{eq:rii-trr-rational} is
rewritten as
\begin{subequations}\label{eq:rii-matrix-lax}
  \begin{equation}
    A^{(t)}\bm\Phi^{(t)}(x)+\kappa_{t+N}\bm\Phi^{(t)}_N(x)=x\left(B^{(t)}\bm\Phi^{(t)}(x)+\bm\Phi^{(t)}_N(x)\right).
  \end{equation}
  Further, let $L^{(t)}_{\mathrm A}$, $L^{(t)}_{\mathrm B}$, and $R^{(t)}$ be
  bidiagonal matrices:
  \begin{gather*}
    L^{(t)}_{\mathrm A}\coloneq
    \begin{pmatrix}
      \kappa_{t}\\
      -\lambda_1 e^{(t)}_1 & \kappa_{t+1}\\
      & -\lambda_2 e^{(t)}_2 & \ddots\\
      && \ddots & \ddots\\
      &&& -\lambda_{N-1} e^{(t)}_{N-1} & \kappa_{t+N-1}
    \end{pmatrix},\quad
    L^{(t)}_{\mathrm B}\coloneq
    \begin{pmatrix}
      1\\
      -e^{(t)}_1 & 1\\
      & -e^{(t)}_2 & \ddots\\
      && \ddots & \ddots\\
      &&& -e^{(t)}_{N-1} & 1
    \end{pmatrix},\\
    R^{(t)}\coloneq
    \begin{pmatrix}
      q^{(t)}_0 & -1\\
      & q^{(t)}_1 & -1\\
      && \ddots & \ddots\\
      &&& \ddots & -1\\
      &&&& q^{(t)}_{N-1}
    \end{pmatrix},
  \end{gather*}
  and $D_{\mathrm q}^{(t)}$, $D_{\mathrm e}^{(t)}$ and $\hat D_{\mathrm e}^{(t)}$
  be diagonal matrices:
  \begin{gather*}
    D_{\mathrm q}^{(t)}\coloneq \diag\left(1+q^{(t)}_0, 1+q^{(t)}_1, \dots, 1+q^{(t)}_{N-1}\right),\\
    D_{\mathrm e}^{(t)}\coloneq \diag\left(1, 1+e^{(t)}_1, \dots, 1+e^{(t)}_{N-1}\right),\quad
    \hat D_{\mathrm e}^{(t)}\coloneq \diag\left(1+e^{(t)}_1, \dots, 1+e^{(t)}_{N-1}, 1\right).
  \end{gather*}
  Then, the spectral transformations \eqref{eq:rii-spectral-transform}
  are written in terms of the rational functions $\{\Phi^{(t)}_n(x)\}_{n=0}^N$ as
  \begin{gather}
    (x-s^{(t)})D^{(t)}_{\mathrm q}\bm\Phi^{(t+1)}(x)=(x-\kappa_t)\left(R^{(t)}\bm\Phi^{(t)}(x)-\bm\Phi^{(t)}_N(x)\right),\\
    D_{\mathrm e}^{(t)}\bm\Phi^{(t)}(x)=(x-\kappa_t)^{-1}\left(xL^{(t)}_{\mathrm B}-L^{(t)}_{\mathrm A}\right)\bm\Phi^{(t+1)}(x).
  \end{gather}
\end{subequations}
Equations \eqref{eq:rii-matrix-lax} yield
\begin{align*}
  &\phantom{{}={}}\,x\left(B^{(t+1)}\bm\Phi^{(t+1)}(x)+\bm\Phi^{(t+1)}_N(x)\right)-A^{(t+1)}\bm\Phi^{(t+1)}(x)-\kappa_{t+N}\bm\Phi^{(t+1)}_N(x)\\
  &=
  \begin{multlined}[t]
    x\left(\left(-D_{\mathrm q}^{(t+1)}L_{\mathrm B}^{(t+1)}(D_{\mathrm q}^{(t+1)})^{-1}R^{(t+1)}+D_{\mathrm q}^{(t+1)}D_{\mathrm e}^{(t+1)}\right)\bm\Phi^{(t+1)}(x)+\bm\Phi^{(t+1)}_N(x)\right)\\
    -\left(-D_{\mathrm q}^{(t+1)}L_{\mathrm A}^{(t+1)}(D_{\mathrm q}^{(t+1)})^{-1}R^{(t+1)}+s^{(t+1)}D_{\mathrm q}^{(t+1)}D_{\mathrm e}^{(t+1)}\right)\bm\Phi^{(t+1)}(x)-\kappa_{t+N}\bm\Phi^{(t+1)}_N(x)
  \end{multlined}\\
  &=
  \begin{multlined}[t]
    x\left(\left(-\hat D_{\mathrm e}^{(t)}R^{(t)}(D_{\mathrm e}^{(t)})^{-1}L_{\mathrm B}^{(t)}+D_{\mathrm q}^{(t)}\hat D_{\mathrm e}^{(t)}\right)\bm\Phi^{(t+1)}(x)+\bm\Phi^{(t+1)}_N(x)\right)\\
    -\left(-\hat D_{\mathrm e}^{(t)}R^{(t)}(D_{\mathrm e}^{(t)})^{-1}L_{\mathrm A}^{(t)}+s^{(t)}D_{\mathrm q}^{(t)}\hat D_{\mathrm e}^{(t)}\right)\bm\Phi^{(t+1)}(x)-\kappa_{t+N}\bm\Phi^{(t+1)}_N(x)
  \end{multlined}\\
  &=\bm 0.
\end{align*}
Hence, the compatibility condition for \eqref{eq:rii-matrix-lax}, i.e.
the matrix form of the monic type finite \Rii chain, is given by
\begin{align*}
  A^{(t+1)}
  &=-D_{\mathrm q}^{(t+1)}L_{\mathrm A}^{(t+1)}(D_{\mathrm q}^{(t+1)})^{-1}R^{(t+1)}+s^{(t+1)}D_{\mathrm q}^{(t+1)}D_{\mathrm e}^{(t+1)}\\
  &=-\hat D_{\mathrm e}^{(t)}R^{(t)}(D_{\mathrm e}^{(t)})^{-1}L_{\mathrm A}^{(t)}+s^{(t)}D_{\mathrm q}^{(t)}\hat D_{\mathrm e}^{(t)},\\
  B^{(t+1)}
  &=-D_{\mathrm q}^{(t+1)}L_{\mathrm B}^{(t+1)}(D_{\mathrm q}^{(t+1)})^{-1}R^{(t+1)}+D_{\mathrm q}^{(t+1)}D_{\mathrm e}^{(t+1)}\\
  &=-\hat D_{\mathrm e}^{(t)}R^{(t)}(D_{\mathrm e}^{(t)})^{-1}L_{\mathrm B}^{(t)}+D_{\mathrm q}^{(t)}\hat D_{\mathrm e}^{(t)}.
\end{align*}
This leads to
\begin{gather*}
  A^{(t+1)}=\hat D_{\mathrm e}^{(t)}R^{(t)}(D_{\mathrm q}^{(t)}D_{\mathrm e}^{(t)})^{-1}A^{(t)}(R^{(t)})^{-1}D_{\mathrm q}^{(t)},\\
  B^{(t+1)}=\hat D_{\mathrm e}^{(t)}R^{(t)}(D_{\mathrm q}^{(t)}D_{\mathrm e}^{(t)})^{-1}B^{(t)}(R^{(t)})^{-1}D_{\mathrm q}^{(t)},
\end{gather*}
and
\begin{equation*}
  xB^{(t+1)}-A^{(t+1)}=\hat D_{\mathrm e}^{(t)}R^{(t)}(D_{\mathrm q}^{(t)}D_{\mathrm e}^{(t)})^{-1}\left(xB^{(t)}-A^{(t)}\right)(R^{(t)})^{-1}D_{\mathrm q}^{(t)}.
\end{equation*}
The last equation implies that the generalized eigenvalues
of the tridiagonal matrix pencil $(A^{(t)}, B^{(t)})$ are
conserved under the time evolution.

\section{Generalized eigenvalue algorithm}
\label{sec:gener-eigenv-algor}

\subsection{Subtraction-free form of the monic type \Rii chain}

In Section~\ref{sec:rii-polynomials-rii}, we have presented
the convergence theorem for the monic type finite \Rii chain
(Theorem~\ref{thm:rii-converge}).
This theorem allows us to design a generalized eigenvalue algorithm
for tridiagonal matrix pencils.

Consider a pair of tridiagonal matrices of order $N$ as input:
\begin{equation}\label{eq:rii-input}
  A=
  \begin{pmatrix}
    a_{0, 0} & a_{0, 1}\\
    a_{1, 0} & a_{1, 1} & a_{1, 2}\\
    & a_{2, 1} & \ddots & \ddots\\
    && \ddots & \ddots & a_{N-2, N-1}\\
    &&& a_{N-1, N-2} & a_{N-1, N-1}
  \end{pmatrix},\quad
  B=
  \begin{pmatrix}
    b_{0, 0} & b_{0, 1}\\
    b_{1, 0} & b_{1, 1} & b_{1, 2}\\
    & b_{2, 1} & \ddots & \ddots\\
    && \ddots & \ddots & b_{N-2, N-1}\\
    &&& b_{N-1, N-2} & b_{N-1, N-1}
  \end{pmatrix}.
\end{equation}
Suppose that all the subdiagonal elements $b_{0, 1}, b_{1, 2}, \dots, b_{N-2, N-1}$
and $b_{1, 0}, b_{2, 0}, \dots, b_{N-1, N-2}$ of the matrix $B$ are nonzero,
and all the leading principal minors of the matrix $B$ are nonzero.
Then, the transformation
\begin{equation*}
  A^{(0)}\coloneq V_1U A U^{-1}V_2,\quad
  B^{(0)}\coloneq V_1U B U^{-1}V_2
\end{equation*}
gives the initial matrix pencil of the form \eqref{eq:rii-tridiagonal-finite}
for the monic type finite \Rii chain, where
\begin{gather*}
  U\coloneq \diag(1, b_{0, 1}, b_{0, 1}b_{1, 2}, \dots, b_{0, 1}b_{1, 2}\dots b_{N-2, N-1}),\\
  V_1\coloneq \diag\left((\det B_1)^{-1}, (\det B_2)^{-1}, \dots, (\det B_N)^{-1}\right),\quad
  V_2\coloneq \diag(1, \det B_1, \det B_2, \dots, \det B_{N-1}),
\end{gather*}
and $B_n$ is the $n$-th order leading principal submatrix of the matrix $B$.
Namely, the elements of $A^{(0)}$ and $B^{(0)}$ are computed by
\begin{equation}\label{eq:rii-transform-init}
  v^{(0)}_n \coloneq a_{n, n}\frac{\det B_n}{\det B_{n+1}},\quad
  w^{(0)}_n \coloneq b_{n-1, n}b_{n, n-1}\frac{\det B_{n-1}}{\det B_{n+1}},\quad
  \kappa_n\coloneq \frac{a_{n, n+1}}{b_{n, n+1}},\quad
  \lambda_n\coloneq \frac{a_{n, n-1}}{b_{n, n-1}}.
\end{equation}
Note that, if $n$ is large, an overflow may occur when one computes
$\det B_n$ directly.
The values $\frac{\det B_n}{\det B_{n+1}}$ and $\frac{\det B_{n-1}}{\det B_{n+1}}$
should be computed by the LU decomposition.
Next, by the relation \eqref{eq:rii},
``decompose'' the matrix pencil $(A^{(0)}, B^{(0)})$
to the variables of the monic type finite \Rii chain:
\begin{subequations}\label{eq:rii-init-values}
  \begin{alignat}{2}
    &e^{(0)}_0\coloneq 0,\quad e^{(0)}_N\coloneq 0,\\
    &\tilde e^{(0)}_n\coloneq \frac{w^{(0)}_n}{q^{(0)}_{n-1}},\quad
    e^{(0)}_n\coloneq\tilde e^{(0)}_n\frac{1+q^{(0)}_{n-1}}{1+q^{(0)}_n},\quad
    &\quad&n=1, 2, \dots, N-1,\\
    &q^{(0)}_n\coloneq \frac{v^{(0)}_n-s^{(0)}(1+w^{(0)}_n)-(s^{(0)}-\lambda_n)\tilde e^{(0)}_n}{s^{(0)}-\kappa_n},\quad
    &&n=0, 1, \dots, N-1.
  \end{alignat}
\end{subequations}
Notice that
the initial matrix pencil $(A^{(0)}, B^{(0)})$ does not fix
the values of the parameters $s^{(0)}$ and $\kappa_{N-1}$.
We must choose the parameters $s^{(0)}$ and $\kappa_{N-1}$ appropriately.
We will discuss how to choose the parameters in the end of this subsection.
After that, compute the time evolution of the monic type finite \Rii chain by using
\eqref{eq:rii-finite} iteratively; i.e., for each $t \ge 0$, compute
\begin{subequations}
  \begin{gather}
    e^{(t+1)}_0\coloneq 0,\quad e^{(t+1)}_N \coloneq 0,\\
    e^{(t+1)}_n\coloneq e^{(t)}_n \frac{q^{(t)}_n}{q^{(t+1)}_{n-1}}\frac{1+q^{(t+1)}_{n-1}}{1+q^{(t+1)}_n}\frac{1+e^{(t)}_{n+1}}{1+e^{(t)}_n},\quad
    n=1, 2, \dots, N-1,\\
    \begin{multlined}[b]
      q^{(t+1)}_n\coloneq (s^{(t+1)}-\kappa_{t+n+1})^{-1}\Bigg((s^{(t+1)}-\kappa_{t+n})q^{(t)}_n\frac{1+e^{(t)}_{n+1}}{1+e^{(t)}_n}-(s^{(t+1)}-\lambda_n)e^{(t+1)}_n\frac{1+q^{(t+1)}_n}{1+q^{(t+1)}_{n-1}}\\
      +(s^{(t+1)}-\lambda_{n+1})e^{(t)}_{n+1}-(s^{(t+1)}-s^{(t)})(1+q^{(t)}_n)(1+e^{(t)}_{n+1})\Bigg),
    \end{multlined}\nonumber\\
    \hspace{33em}n=0, 1, \dots, N-1.\label{eq:rii-time-evol-q}
  \end{gather}
\end{subequations}
Here, we also have to choose the parameters $s^{(t+1)}$ and $\kappa_{t+N}$
for computing the above recurrence equations.
From the results in Subsection~\ref{sec:finite-dimens-case},
we can see that if the absolute values of all the subdiagonal elements
$\lambda_n w^{(t)}_n$ and $w^{(t)}_n$ of
the matrix pencil $(A^{(t)}, B^{(t)})$ become
sufficiently small at a time $t$, then the values
$(s^{(t)}-\kappa_{t+n})q^{(t)}_n+s^{(t)}$ give
the generalized eigenvalues of the initial tridiagonal matrix pencil $(A, B)$.
In general, however, equation \eqref{eq:rii-time-evol-q} requires subtraction
operations, which may degrade the accuracy by the loss of significant digits.
A subtraction-free form of the monic type finite \Rii chain may resolve the problem.

Let us introduce an auxiliary variable
\begin{equation*}
  d^{(t+1)}_n=\frac{(s^{(t+1)}-\kappa_{t+n+1})q^{(t+1)}_n-(s^{(t+1)}-\lambda_{n+1})e^{(t)}_{n+1}}{1+e^{(t)}_{n+1}},\quad
  n=0, 1, \dots, N-1.
\end{equation*}
This is an analogue of the auxiliary variable \eqref{eq:dqds-aux} introduced
in the dqds algorithm.
Then, the subtraction-free form is derived as
\begin{subequations}\label{eq:rii-subtraction-free}
  \begin{gather}
    \begin{multlined}[b]
      d^{(t+1)}_0\coloneq(s^{(t)}-\kappa_t)q^{(t)}_0-(s^{(t+1)}-s^{(t)}),\quad
      d^{(t+1)}_n\coloneq d^{(t+1)}_{n-1}\frac{q^{(t)}_n}{q^{(t+1)}_{n-1}}-(s^{(t+1)}-s^{(t)})(1+q^{(t)}_n),\\
      n=1, 2, \dots, N-1,
    \end{multlined}\label{eq:rii-subtraction-free-d}\\
    q^{(t+1)}_n
    \coloneq\frac{(s^{(t+1)}-\lambda_{n+1})e^{(t)}_{n+1}+d^{(t+1)}_n(1+e^{(t)}_{n+1})}{s^{(t+1)}-\kappa_{t+n+1}},\quad
    n=0, 1, \dots, N-1,\\
    e^{(t+1)}_0\coloneq0,\quad
    e^{(t+1)}_n
    \coloneq e^{(t)}_n\frac{q^{(t)}_n}{q^{(t+1)}_{n-1}}\frac{1+q^{(t+1)}_{n-1}}{1+q^{(t+1)}_n}\frac{1+e^{(t)}_{n+1}}{1+e^{(t)}_n},\quad n=1, 2, \dots, N-1,\quad
    e^{(t+1)}_N\coloneq0.
  \end{gather}
\end{subequations}

From the spectral transformations \eqref{eq:rii-spectral-transform},
we have
\begin{gather*}
  -(1+e^{(t-1)}_{n+1})\varphi^{(t-1)}_{n+1}(s^{(t)})=\left((s^{(t)}-\kappa_{t+n})q^{(t)}_n-(s^{(t)}-\lambda_{n+1})e^{(t-1)}_{n+1}\right)\varphi^{(t)}_n(s^{(t)})\\
  \Rightarrow\quad
  d^{(t)}_n
  =-\frac{\varphi^{(t-1)}_{n+1}(s^{(t)})}{\varphi^{(t)}_n(s^{(t)})}
  =\frac{\tau^{n, n, t}_n\sigma^{n, n+1, t}_{n+1}}{\tau^{n+1, n+1, t-1}_{n+1}\tau^{n-1, n, t+1}_n},
\end{gather*}
where
\begin{equation*}
  \sigma^{k, l, t}_{n}\coloneq |\mu^{k+1, l, t-1}_{i+j+1}-s^{(t)}\mu^{k+1, l, t-1}_{i+j}|_{0 \le i, j \le n-1}.
\end{equation*}
In addition, we already have the expression of $q^{(t)}_n$ \eqref{eq:rii-q}.
Hence, we obtain a sufficient condition for computing
the recurrence equation \eqref{eq:rii-subtraction-free}
without subtraction operations except the shift terms
in \eqref{eq:rii-subtraction-free-d}: for all $n$ and $t$,
\begin{subequations}\label{eq:rii-positive-condition}
  \begin{gather}
    w^{(0)}_n>0,\label{eq:rii-positive-condition-1}\\
    (-1)^n\varphi^{(t)}_n(s^{(t)})=\det(A^{(t)}_n-s^{(t)}B^{(t)}_n)>0,\label{eq:rii-positive-condition-2}\\
    (-1)^n\varphi^{(t)}_n(s^{(t+1)})=\det(A^{(t)}_n-s^{(t+1)}B^{(t)}_n)>0,\label{eq:rii-positive-condition-3}\\
    s^{(t)}>\kappa_{t+n},\quad s^{(t)}>\lambda_n.
  \end{gather}
\end{subequations}
By \eqref{eq:rii-transform-init},
if the input tridiagonal matrix $B$ is a real symmetric positive (or negative)
definite matrix, then the condition \eqref{eq:rii-positive-condition-1} is satisfied.
Further, assume that
the generalized eigenvalues $x_0, x_1, \dots, x_{N-1}$ of the input tridiagonal
matrix pencil $(A, B)$ are all real and simple,
the matrix $A$ is a real matrix
and the conditions $w^{(t)}_n>0$ and $\kappa_{t+n-1}=\lambda_n$ are
satisfied for $n=1, 2, \dots, N-1$ at some time $t$.
Then, it is shown that if the parameter $s^{(t)}$ is chosen as
$s^{(t)}<\min\{x_0, x_1, \dots, x_{N-1}\}$,
the condition \eqref{eq:rii-positive-condition-2} is satisfied.
The condition \eqref{eq:rii-positive-condition-3} is also satisfied with
$s^{(t+1)}<\min\{x_0, x_1, \dots, x_{N-1}\}$.
From Theorem~\ref{thm:rii-converge},
if $s^{(t)}$ is chosen as close as possible to $\min\{x_0, x_1, \dots, x_{N-1}\}$
under the conditions \eqref{eq:rii-positive-condition},
the convergence speed is accelerated.

By summarizing this subsection,
Algorithm~\ref{alg:riigev} is proposed as a new generalized eigenvalue algorithm
for tridiagonal matrix pencils based on the monic type finite \Rii chain.

\begin{algorithm}[htbp]
  \caption{The proposed generalized eigenvalue algorithm based on the monic type finite \Rii chain}\label{alg:riigev}
  \begin{algorithmic}[1]
    \Function{GEVRII}{$A, B$}
    \Comment{$A$ and $B$ are tridiagonal matrices of the form \eqref{eq:rii-input}}
    \State Compute $\{v^{(0)}_n\}_{n=0}^{N-1}$, $\{w^{(0)}_n\}_{n=1}^{N-1}$,
    $\{\kappa_n\}_{n=0}^{N-2}$, and $\{\lambda_n\}_{n=1}^{N-1}$ by \eqref{eq:rii-transform-init}
    \State Set the parameters $s^{(0)}$ and $\kappa_{N-1}$ appropriately
    \Comment{See Theorem~\ref{thm:rii-converge} and the condition \eqref{eq:rii-positive-condition}}
    \State Compute $\{q^{(0)}_n\}_{n=0}^{N-1}$ and $\{e^{(0)}_n\}_{n=0}^N$ by \eqref{eq:rii-init-values}
    \State $t\coloneq 0$
    \Repeat
    \State Set the parameters $s^{(t+1)}$ and $\kappa_{t+N}$ appropriately
    \Comment{See Theorem~\ref{thm:rii-converge} and the condition \eqref{eq:rii-positive-condition}}
    \State Compute $\{q^{(t+1)}_n\}_{n=0}^{N-1}$ and $\{e^{(t+1)}_n\}_{n=0}^N$ by \eqref{eq:rii-subtraction-free}
    \State $t \coloneq t+1$
    \For{$n=1, 2, \dots, N-1$}
    \State$w^{(t)}_n\coloneq q^{(t)}_{n-1}e^{(t)}_n\frac{1+q^{(t)}_n}{1+q^{(t)}_{n-1}}$
    \EndFor
    \Until the absolute values of $w^{(t)}_n$ and $\lambda_n w^{(t)}_n$
    are sufficiently small for all $n=1, 2, \dots, N-1$
    \State \Return $\{(s^{(t)}-\kappa_{t+n})q^{(t)}_n+s^{(t)}\}_{n=0}^{N-1}$
    \EndFunction
  \end{algorithmic}
\end{algorithm}

\subsection{Numerical examples}

We shall give numerical examples.
To construct test problems with known generalized eigenvalues,
let us consider the monic finite orthogonal polynomials $\{p_n(x)\}_{n=0}^{N}$
defined by
\begin{gather*}
  p_{n+1}(x)\coloneq \left(x-\frac{N-1}{2}\right)p_n(x)-\frac{n(N-n)}{4}p_{n-1}(x),\quad n=0, 1, \dots, N-1,
\end{gather*}
with $p_{-1}(x)\coloneq 0$ and $p_0(x)\coloneq 1$.
The polynomials $\{p_n(x)\}_{n=0}^N$ are the monic Krawtchouk polynomials
with a special parameter and
it is well known that the Krawtchouk polynomials are orthogonal
on $x=0, 1, \dots, N-1$ with respect to the binomial distribution
\cite{koekoek1998ash}.
This means that the tridiagonal matrix of order $N$
\begin{equation*}
  \tilde K_N\coloneq
  \begin{pmatrix}
    (N-1)/2 & 1\\
    (N-1)/4 & (N-1)/2 & 1\\
    & 2(N-2)/4 & (N-1)/2 & 1\\
    && 3(N-3)/4 & \ddots & \ddots\\
    &&& \ddots & \ddots & 1\\
    &&&& (N-1)/4 & (N-1)/2
  \end{pmatrix} \in \mathbb R^{N\times N}
\end{equation*}
has the eigenvalues $0, 1, \dots, N-1$.
The symmetric tridiagonal matrix
\begin{equation*}
  K_N\coloneq
  \begin{pmatrix}
    (N-1)/2 & \sqrt{(N-1)/4}\\
    \sqrt{(N-1)/4} & (N-1)/2 & \sqrt{2(N-2)/4}\\
    & \sqrt{2(N-2)/4} & (N-1)/2 & \sqrt{3(N-3)/4}\\
    && \sqrt{3(N-3)/4} & \ddots & \ddots\\
    &&& \ddots & \ddots & \sqrt{(N-1)/4}\\
    &&&& \sqrt{(N-1)/4} & (N-1)/2
  \end{pmatrix} \in \mathbb R^{N\times N}
\end{equation*}
is similar to $\tilde K_N$.
Hence, it is readily shown that the tridiagonal matrix pencil
$(K_N+2I_N, K_N+I_N)$ has the generalized eigenvalues $(n+1)/n$, $n=1, 2, \dots, N$.

The following experiments were run on a Linux PC with
kernel 3.7.4 and gcc 4.7.2 on
Intel Core i5 760 2.80 GHz CPU and 4 GB memory.
All the computations were performed in double precision
and the stopping criterion (line 13 in Algorithm~\ref{alg:riigev}) was
$|w^{(t)}_n|<10^{-20}$ and $|\lambda_n w^{(t)}_n|<10^{-20}$
for all $n=1, 2, \dots, N-1$.

\begin{example}
  The first example is the case of $N=5$:
  \begin{equation*}
    K_5=
    \begin{pmatrix}
      2 & 1\\
      1 & 2 & \sqrt{3/2}\\
      & \sqrt{3/2} & 2 & \sqrt{3/2}\\
      && \sqrt{3/2} & 2 & 1\\
      &&& 1 & 2
    \end{pmatrix}.
  \end{equation*}
  The generalized eigenvalues of the matrix pencil $(K_5+2I_5, K_5+I_5)$
  are $2$, $3/2$, $4/3$, $5/4$, and $6/5$.
  By this example, we will observe the behaviour of the variables
  of the monic type finite \Rii chain and
  confirm that the proposed algorithm computes the generalized eigenvalues
  of a given matrix pencil and its convergence speed depends on the
  parameters $s^{(t)}$ and $\kappa_{t+n}$.

  \begin{figure}[tbp]
    \begin{center}
      \includegraphics[scale=.8]{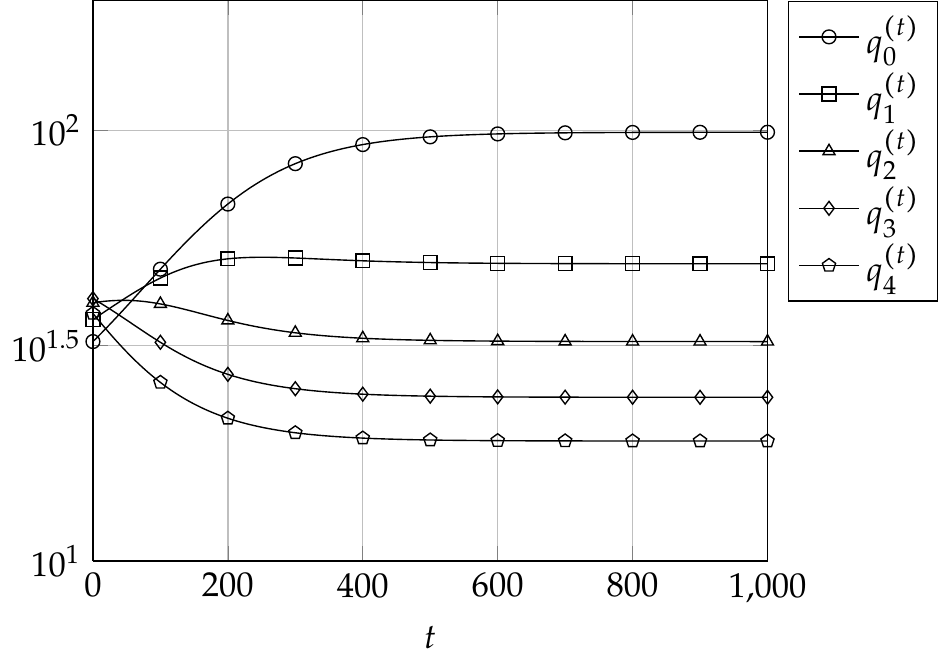}
      \quad
      \includegraphics[scale=.8]{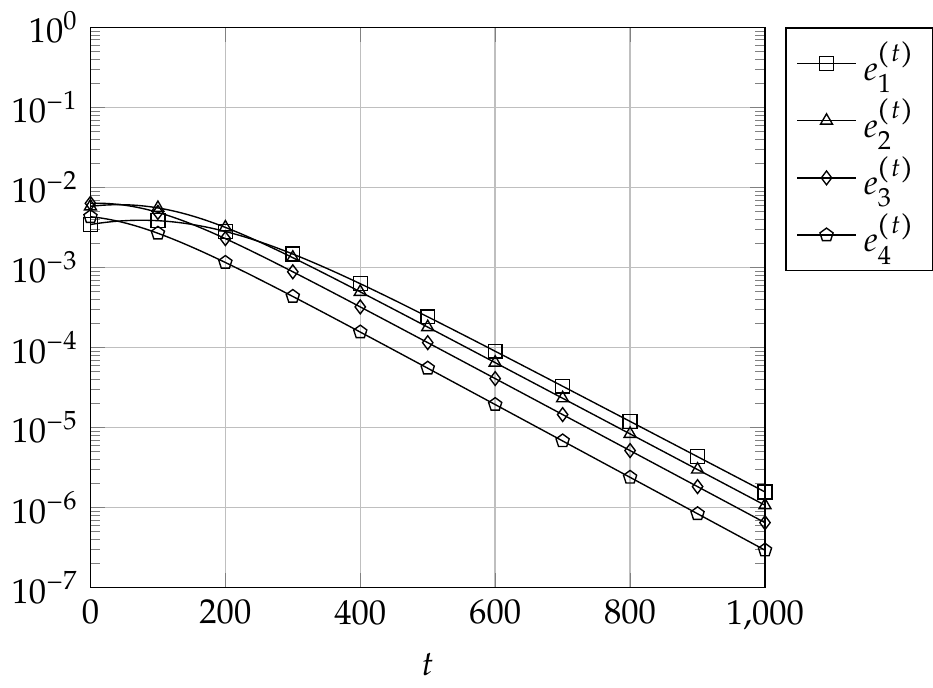}\\[2em]
      \includegraphics[scale=.8]{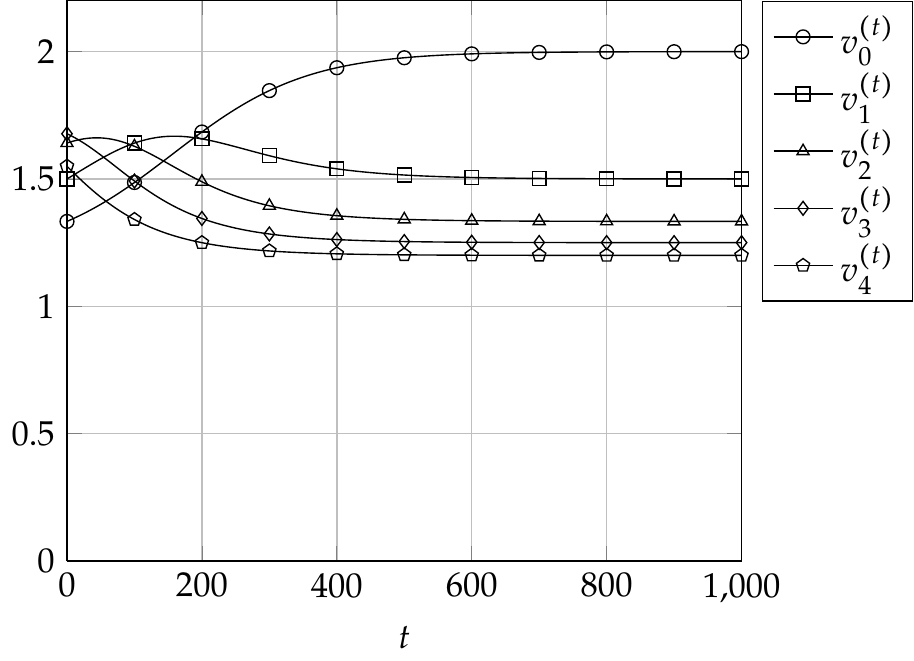}
      \quad
      \includegraphics[scale=.8]{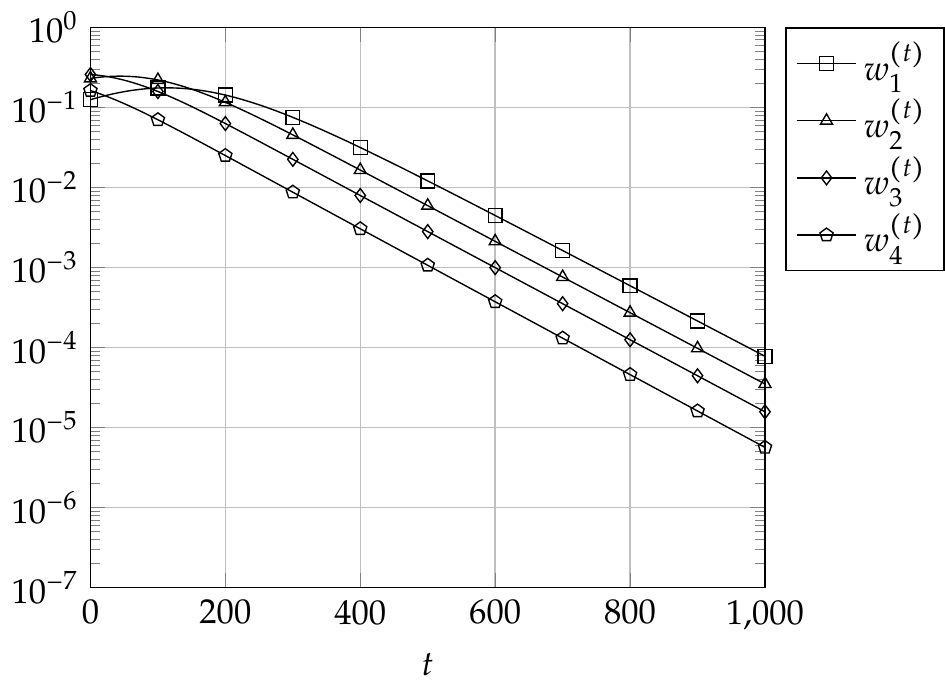}
    \end{center}
    \caption{The behaviour of the variables of the monic type finite \Rii chain
      for the input tridiagonal matrix pencil $(K_5+2I_5, K_5+I_5)$
      with the parameters $s^{(t)}=1.01$ for all $t \ge 0$ and $\kappa_n=1$ for all $n \ge 4$.}
    \label{fig:s=1.01}
  \end{figure}

  Figure~\ref{fig:s=1.01} shows the result with the parameters
  $s^{(t)}=1.01$ for all $t \ge 0$ and $\kappa_n=1$ for all $n \ge 4$,
  where $q^{(t)}_n$ and $e^{(t)}_n$ are the variables of the monic type
  \Rii chain, $v^{(t)}_n$ are the diagonal elements of $A^{(t)}$,
  and $w^{(t)}_n$ are the subdiagonal elements of $B^{(t)}$
  (see equations \eqref{eq:rii-v} and \eqref{eq:rii-w}).
  We can confirm that $v^{(t)}_n$ and $w^{(t)}_n$ converge linearly to
  the eigenvalues and zero, respectively.
  Since the shift parameter $s^{(t)}$ is not so close
  to the minimal eigenvalue $6/5=1.2$,
  the stopping criterion is satisfied at $t=4605$.

  \begin{figure}[tbp]
    \begin{center}
      \includegraphics[scale=.8]{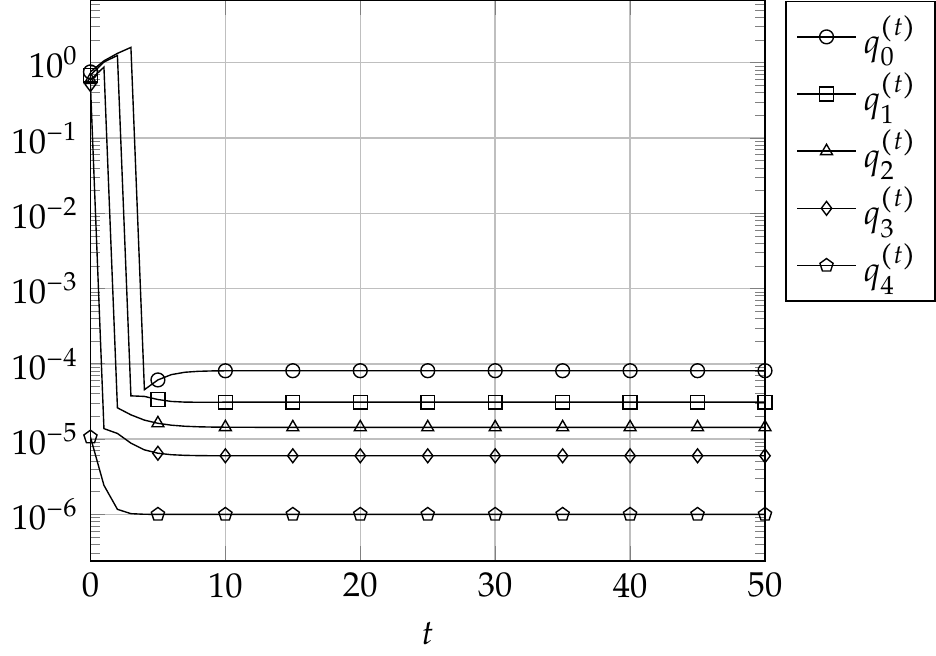}
      \quad
      \includegraphics[scale=.8]{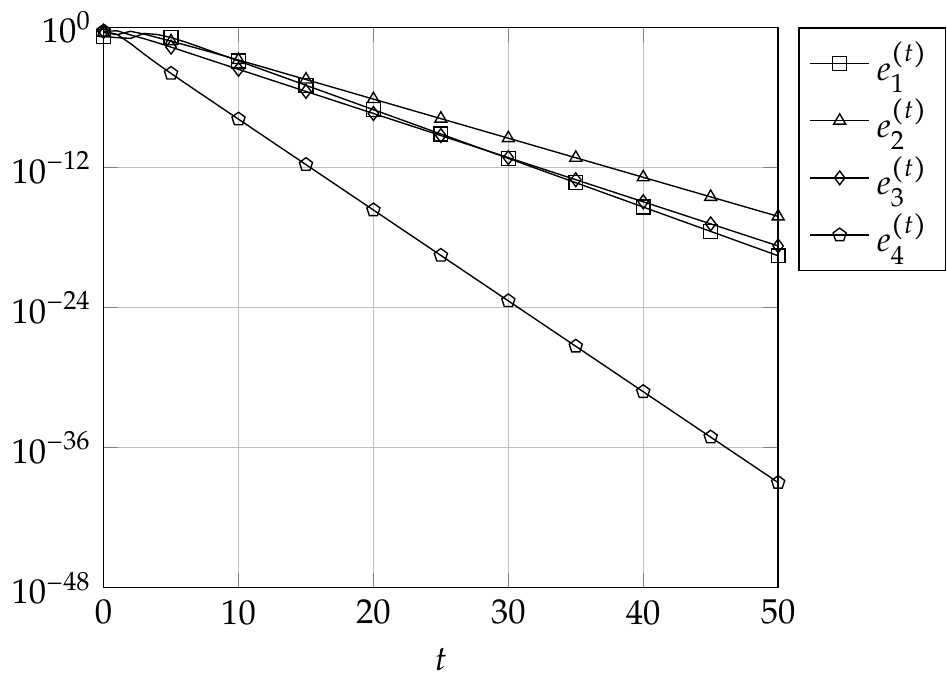}\\[2em]
      \includegraphics[scale=.8]{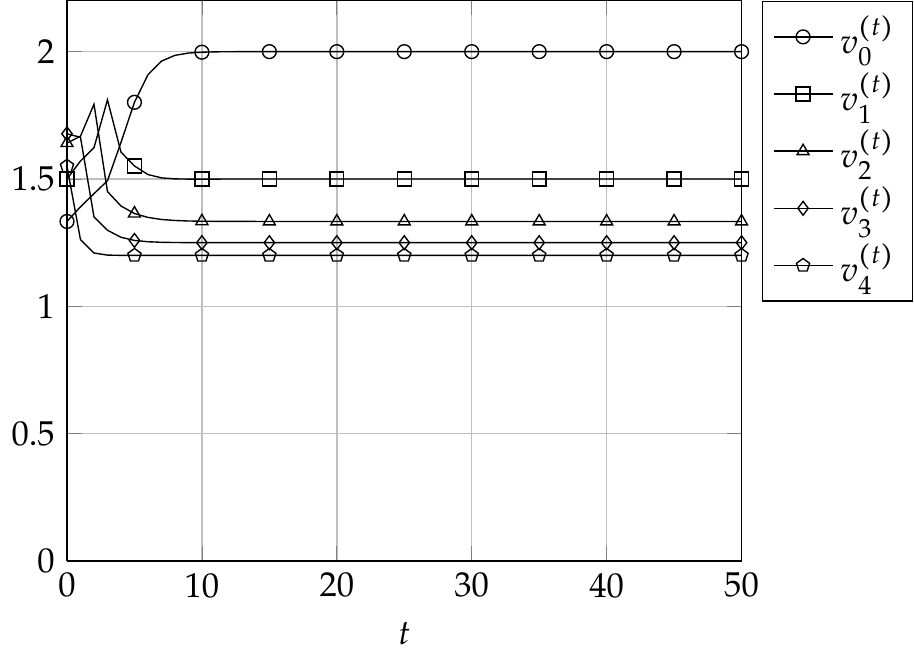}
      \quad
      \includegraphics[scale=.8]{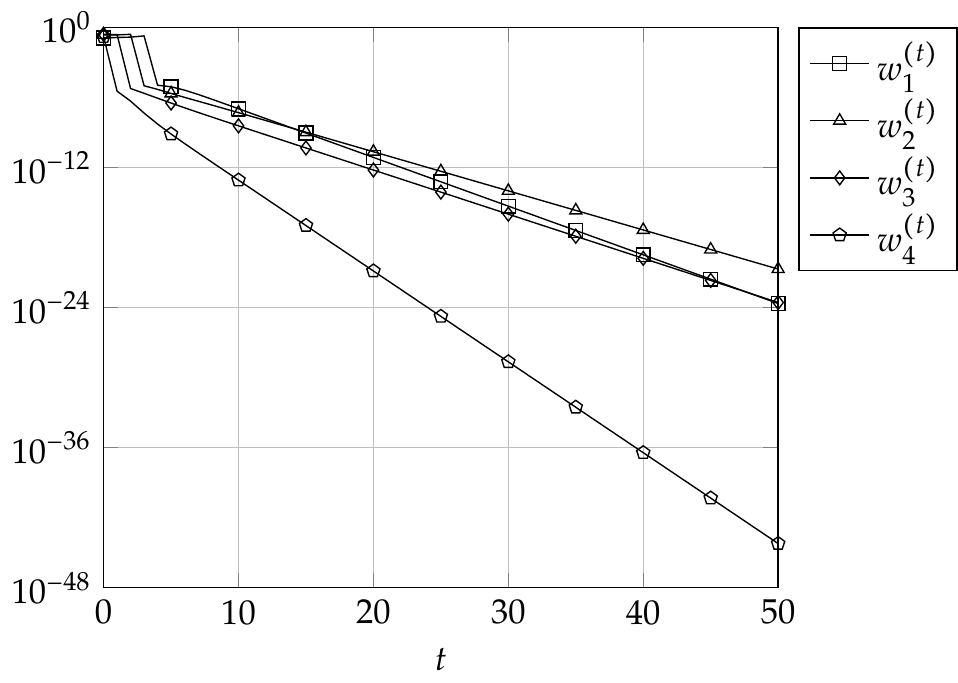}
    \end{center}
    \caption{The behaviour of the variables of the monic type finite \Rii chain
      for the input tridiagonal matrix pencil $(K_5+2I_5, K_5+I_5)$
      with the parameters $s^{(t)}=1.19$ for all $t \ge 0$ and $\kappa_{n}=-10000$
      for all $n \ge 4$.}
    \label{fig:s=1.19}
  \end{figure}
  \begin{table}[tbp]
    \caption{The eigenvalues computed by Algorithm~\ref{alg:riigev}.
      The parameters are $s^{(t)}=1.19$ and $\kappa_{n}=-10000$ for all $t \ge 0$
      and $n \ge 4$.}
    \label{tab:computed-eigenvalues}
    \begin{center}
      \begin{tabular}{cc}
        \hline
        Computed eigenvalues & True eigenvalues\\\hline
        1.9999999999999998&	2.0000000000000000\\
        1.4999999999999991&	1.5000000000000000\\
        1.3333333333333335&	1.3333333333333333\\
        1.2500000000000000&	1.2500000000000000\\
        1.2000000000000000&	1.2000000000000000\\
        \hline
      \end{tabular}
    \end{center}
  \end{table}

  Figure~\ref{fig:s=1.19} shows the result with more suitable parameters:
  $s^{(t)}=1.19$ for all $t \ge 0$ and $\kappa_{n}=-10000$ for all $n \ge 4$.
  The convergence speed is much faster than the former example;
  the stopping criterion is satisfied at $t=48$.
  Table~\ref{tab:computed-eigenvalues} shows the computed eigenvalues.
\end{example}

\begin{example}\label{ex:large}
  Next, the test cases for $N=512$, $1024$, $2048$, $4096$, $8192$ were computed
  by two methods.
  By these examples, we will compare the computation time
  and the accuracy of the proposed algorithm
  with a routine called \texttt{DSYGV} in LAPACK 3.4.2 \cite{lapack}.
  \texttt{DSYGV} computes the generalized eigenvalues of a given matrix pencil
  $(A, B)$ in double precision,
  where $A$ is real symmetric and $B$ is real symmetric and positive definite.
  Internally,
  \texttt{DSYGV} computes
  the Cholesky factorization $B=LL^{\mathrm{T}}$, where $L$ is
  a lower triangular matrix, transforms the generalized eigenvalue problem
  $A\bm \varphi=xB \bm \varphi$ to the eigenvalue problem
  $L^{-1}AL^{-\mathrm{T}}(L^{\mathrm{T}} \bm\varphi)=x(L^{\mathrm{T}} \bm \varphi)$
  and solves the eigenvalue problem.
  We should remark that, even if $A$ and $B$ are both tridiagonal,
  $L^{-1}AL^{-\mathrm{T}}$ is a dense matrix in general.
  Hence, we expect that \texttt{DSYGV} spends much time for
  large problems.
  On the other hand, the proposed algorithm preserves the tridiagonal form
  of the matrices $A^{(t)}$ and $B^{(t)}$.
  The proposed algorithm will thus compute the generalized eigenvalues
  of tridiagonal matrix pencils fast and accurately for large problems.

  \begin{table}[tbp]
    \caption{The results of the computation by Algorithm~\ref{alg:riigev}
    for the generalized eigenvalue problems of $(K_N+2I_N, K_N+I_N)$.}
    \label{tab:result-rii}
    \begin{center}
      \begin{tabular}{c|ccccc}
        \hline
        Problem size ($N$) & 512 & 1024 & 2048 & 4096 & 8192\\\hline
        Computation time [sec.] &  0.0958 & 0.392 & 1.58 & 6.24 & 24.6\\
        Maximum relative error &  $3.109\times 10^{-15}$ & $3.405\times 10^{-15}$ & $1.776\times 10^{-15}$ & $3.701 \times 10^{-15}$ & $2.043 \times 10^{-14}$\\
        Average relative error & $1.344\times 10^{-16}$ & $1.211 \times 10^{-16}$ & $1.154 \times 10^{-16}$ & $1.072 \times 10^{-16}$ & $1.129 \times 10^{-16}$\\\hline
      \end{tabular}
    \end{center}
    \caption{The results of the computation by \texttt{DSYGV} in LAPACK
    for the generalized eigenvalue problems of $(K_N+2I_N, K_N+I_N)$.}
    \label{tab:result-dsygv}
    \begin{center}
      \begin{tabular}{c|ccccc}
        \hline
        Problem size ($N$) & 512 & 1024 & 2048 & 4096 & 8192\\\hline
        Computation time [sec.] &  0.162 & 1.92 & 30.3 & 307 & 2400\\
        Maximum relative error & $3.664\times 10^{-15}$ & $6.815\times 10^{-15}$ & $1.304 \times 10^{-14}$ & $1.684 \times 10^{-14}$ & $2.949\times 10^{-14}$\\
        Average relative error & $6.673\times 10^{-16}$ & $8.469\times 10^{-16}$ & $1.035 \times 10^{-15}$ & $1.276 \times 10^{-15}$ & $1.508\times 10^{-15}$\\\hline
      \end{tabular}
    \end{center}
  \end{table}

  Tables~\ref{tab:result-rii} and \ref{tab:result-dsygv} show
  the results of the computation by the proposed algorithm and \texttt{DSYGV},
  respectively.
  The parameters for the proposed algorithm are $s^{(t)}=(N+2)/(N+1)$
  for all $t \ge 0$ and $\kappa_n=-10000$ for all $n \ge N-1$.
  In all the cases, the proposed algorithm is faster and more accurate
  than \texttt{DSYGV}.
  In particular, the proposed algorithm has an advantage in computation time
  for large problems.
  Remark that the techniques called deflation and splitting
  (if $|w^{(t)}_n|$ and $|\lambda_n w^{(t)}_n|$ become sufficiently small
  for some $n$ at a time $t$,
  then the problem can be deflated or split into two problems) were not
  implemented in the program used for the experiments.
  These techniques may further accelerate the proposed algorithm.
\end{example}

\section{Conclusion}
In this paper, we have studied the monic type \Rii chain in detail and
proposed a generalized eigenvalue algorithm for tridiagonal matrix pencils based
on a subtraction-free form of the monic type finite \Rii chain.
It has been shown that, similarly to the dqds algorithm,
the parameter $s^{(t)}$ in the monic type finite \Rii chain plays
the role of the origin shifts to accelerate convergence
and the proposed algorithm computes the generalized
eigenvalues of tridiagonal matrix pencils fast and accurately.

In Example~\ref{ex:large}, the shift parameter $s^{(t)}$ is chosen ideally
and all the conditions \eqref{eq:rii-positive-condition} are satisfied.
However, it is difficult to make this situation in general.
Further improvements are thus required for practical use.
First, in general, the condition for positivity
\eqref{eq:rii-positive-condition} is not sufficient for applications;
the condition does not provide concrete ways to choose the parameters for
general cases.
Second, for applying the proposed algorithm to general (not tridiagonal)
matrix pencils,
a preconditioning called simultaneous tridiagonalization
(see, e.g., \cite{garvey2003stt,sidje2011stt}) is required.
In addition to the improvements, comparisons with traditional methods
should be discussed.
These are left for future work.

\section*{Acknowledgments}
The authors would like to thank Professor Yoshimasa Nakamura and
Professor Alexei Zhedanov for valuable discussions and comments.
This work was supported by JSPS KAKENHI Grant Numbers 11J04105 and 22540224.

%\section*{References}

%% \bibliographystyle{elsarticle-num}
%% \bibliography{kmaeda}

\end{document}